\documentclass[a4paper,reqno,english]{amsart}

\usepackage{tikz-cd}
\usetikzlibrary{calc}
\allowdisplaybreaks
\usepackage{latexsym}

\usetikzlibrary{positioning}
\usepackage{graphicx}
\usepackage{mathrsfs}
\usepackage{amsfonts}
\usepackage{amssymb}
\usepackage{amsmath}
\usepackage{amsthm}
\usepackage{amscd}
\usepackage[all,2cell]{xy}
\usepackage{color}
\usepackage{xcolor}
\usepackage[pagebackref,colorlinks]{hyperref}
\usepackage{enumerate}
\usepackage{bm}
\usepackage{bbm}
\usepackage[normalem]{ulem}
\usepackage{mathrsfs}

\usepackage{rotating}

\usepackage{comment}
\usepackage{hyperref}
\UseAllTwocells \SilentMatrices

\numberwithin{equation}{section}

\usepackage{comment}

\theoremstyle{plain}
\newtheorem{thm}{Theorem}[section]
\newtheorem{prop}[thm]{Proposition}
\newtheorem{cor}[thm]{Corollary}
\newtheorem{lemma}[thm]{Lemma}
\newtheorem{conj}[thm]{Conjecture}
\newtheorem{ques}[thm]{Question}

\theoremstyle{definition}
\newtheorem{deff}[thm]{Definition}
\newtheorem{example}[thm]{Example}

\theoremstyle{remark}

\newcommand{\ssink}{\operatorname{sink}}

\def\g{\gamma}

\def\G{\Gamma}

\def\a{\alpha}
\def\b{\beta}

\newcommand{\K}{\mathsf{k}}

\newcommand{\reg}{\operatorname{reg}}

\newcommand{\rr}{\operatorname{{\bf r}}}

\newcommand{\V}{\mathcal{V}}

\newcommand{\gr}{\operatorname{gr}}

\def \Z{\mathbb Z}
\def\-{\text{-}}

\begin{document}

\title[$K_0^{\gr}$ of weighted Leavitt path algebras]{The graded Grothendieck group $\bm{K}_0^{\gr}$ is full for \\weighted Leavitt path algebras}

\author{Remarl Joseph Damalerio}
\address{Remarl Joseph Damalerio: Mindanao State University - General Santos \& Mindanao State University - Iligan Institute of Technology, Philippines}
\email{rjdamalerio@gmail.com}

\author{Roozbeh Hazrat}
\address{Roozbeh Hazrat: Centre for Research in Mathematics and Data Science\\Western Sydney University, Australia} \email{r.hazrat@westernsydney.edu.au}

\author{Tran Giang Nam}
\address{Tran Giang Nam: Institute of Mathematics, VAST, 18 Hoang Quoc Viet, Nghia Do, Hanoi, Vietnam} 
\email{tgnam@math.ac.vn}

\subjclass[2020]{16W50,16S50,16S88,16E20,19K14}

\keywords{weighted Leavitt path algebra, graded ring, $\Gamma$-monoid, talented monoid, graded Grothendieck group}

\date{\today}

\begin{abstract}  
The Graded Classification Conjecture asserts that the graded Grothendieck group \(K_0^{\mathrm{gr}}\) is a complete invariant for the classes of Leavitt path algebras and graph \(C^*\)-algebras. The conjecture remains open, as neither a proof nor a counterexample is currently known.
In this article, we extend the study of this invariant to the class of vertex-weighted Leavitt path algebras. We show that \(K_0^{\mathrm{gr}}\) distinguishes weighted Leavitt path algebras from ordinary (unweighted) Leavitt path algebras. We further prove that an isomorphism between the graded Grothendieck groups of weighted Leavitt path algebras induces an isomorphism between the corresponding semilattices of vertex-generated ideals. In addition, we show that \(K_0^{\mathrm{gr}}\) classifies the classical Leavitt algebras \(L_\K(n,n+k)\). Next, we prove that \(K_0^{\mathrm{gr}}\) is a full functor on the category of all weighted Leavitt path algebras. Consequently, in the special case where all weights are equal to $1$, we recover the lifting theorem established independently by Arnone and Vas for Leavitt path algebras. This confirms one direction of the Graded Classification Conjecture for Leavitt path algebras.
\end{abstract}

\maketitle

\begin{center}
\emph{To Gene Abrams,
with affection and gratitude.}
\end{center}

\section{Introduction}
Given a graph $E$ and a field $\K$, Abrams and Aranda Pino in \cite{ap:tlpaoag05}, and independently Ara, Moreno, and Pardo in \cite{amp}, introduced the \emph{Leavitt path algebra} $L_\K(E)$, which is a quotient of a path algebra obtained by imposing some additional relations amongst the generators. These Leavitt path algebras generalise the algebras $L_{\K}(1, 1 +k)$ constructed by Leavitt \cite{leav:tmtoar}, as universal rings without invariant basis number, and are also the discrete, purely algebraic version of graph $C^*$-algebras. We refer the reader to \cite{lpabook} and \cite{willie} for a detailed history and overview of Leavitt path algebras. 

The classification of Leavitt path algebras has received a lot of attention, as they are closely related to the classification of graph $C^*$-algebras and those of symbolic dynamics. However, while graph $C^*$-algebras have been classified by means of $K$-theoretic invariants (see, e.g., \cite{Res2006, Rodam95, ERRS21}), an analogous classification for Leavitt path algebras remains conjectural. The Graded  Classification Conjecture for Leavitt path algebras states that the graded Grothendieck group $K^{\operatorname{gr}}_0$,  is a complete invariant for 
these algebras~(\cite{willie, haz2013}, \cite[\S7.3.4]{lpabook}):  
%That is, the unital Leavitt path algebras $L_\K(E)$ and $L_\K(F)$, associated to the graphs $E$ and $F$, respectively, are graded isomorphic if and only if there is a preordered $\mathbb Z$-module isomorphism $K^{\operatorname{gr}}_0(L_\K(E)) \rightarrow K^{\operatorname{gr}}_0(L_\K(F))$, mapping $[L_\K(E)]$ to $[L_\K(F)]$. The other question raised in the same paper was whether the $K^{\operatorname{gr}}_0$-functor is full, that is, a preordered $\mathbb Z$-homomorphism $K^{\operatorname{gr}}_0(L_\K(E)) \rightarrow K^{\operatorname{gr}}_0(L_\K(F))$, mapping $[L_\K(E)]$ to $[L_\K(F)]$, can be lifted to a graded homomorphism $L_\K(E)\rightarrow L_\K(F)$.

\begin{conj}\label{conjalg}
Let $E$ and $F$ be finite graphs and $\K$ a field.
\begin{enumerate}[\upshape(1)]

\item For any order preserving $\Z[x,x^{-1}]$-module homomorphism $\phi: K^{\gr}_0(L_\K(E)) \rightarrow K^{\gr}_0(L_\K(F))$ with 
$\phi([L_\K(E)])=L_\K(F)$, there exists a unital $\mathbb Z$-graded $\K$-homomorphism $\psi: L_\K(E) \rightarrow L_\K(F)$ such that $K^{\gr}_0(\psi) = \phi$.

\medskip 

    \item If in \textnormal{(1)}, the $\phi$ is an isomorphism, then there exists a unital $\mathbb Z$-graded $\K$-isomorphism $\psi: L_\K(E) \rightarrow L_\K(F)$ such that $K^{\gr}_0(\psi) = \phi$.

\end{enumerate}
\end{conj}

G. Arnone in \cite{arnone} and L. Vas in \cite{vas} independently answered part (1) of Conjecture~\ref{conjalg}  in the affirmative, using different approaches. In fact Arnone showed that the lifting map can be diagonal preserving graded $*$-homomorphism, while Vas established the result for graphs that may contain infinitely many edges. Part (2) of the Graded Classification Conjectures remains open, as neither a proof nor a counterexample is currently known. The graded Grothendieck group $K_0^{\gr}$ is closely related to Krieger's dimension group~\cite{Krieger1980}, thereby connecting the classification of  Leavitt path algebras with the classification of shift of finite types in symbolic dynamics~\cite{LindMarcus2}. For a comprehensive account of the graded classification conjectures and their connections to related topics, we refer the reader to \cite{willie}.

Whereas Leavitt path algebras model Leavitt algebras $L_\K(1,1+k)$, weighted Leavitt path algebras were introduced by the second author in \cite{H} to cover the more general Leavitt algebras $L_\K(n,n+k)$ \cite{leav:tmtoar}, which cannot, in general, be realised as ordinary Leavitt path algebras. We also note that weighted Leavitt path algebras are closely connected to sandpile models (see, e.g., \cite{GeneRooz, haznam, haznam1}).
Since weighted Leavitt path algebras are also naturally $\mathbb Z$-graded rings, it is natural to investigate whether the graded Grothendieck group $K_0^{\gr}$  carries structural information about these algebras analogous to that provided by $K_0^{\gr}$ in the ordinary Leavitt path algebra setting.

In this article, we show that $K^{\gr}_0$ detects the presence of weights, thereby distinguishing weighted Leavitt path algebras from their unweighted counterparts. This is not the case when one uses the ordinary $K_0$-group. Indeed, for $k,l, n\in \mathbb N$, there are order isomorphisms
\[K_0(L_\K(1,1+k))\cong K_0(L_\K(n,n+k))\cong K_0(L_\K(n+l,n+k+l))\cong \mathbb Z_{k}.\]
Nevertheless, the algebras are pairwise non-isomorphic: the first is an ordinary Leavitt path algebra, while the others arise only as weighted Leavitt path algebras.

For a vertex-weighted graph $(E,w)$, we show that $K_0^{\gr}$ is a full functor, that means, Conjecture~\ref{conjalg}(1) holds for vertex-weighted graphs. By specializing to the case where all weights are equal to one, we recover the corresponding results of Arnone and Vas mentioned above. Furthermore, we obtain that every order preserving $\mathbb{Z}$-module isomorphism between $K_0^{\gr}(L_{\K}(E, w))$ and $K_0^{\gr}(L_{\K}(F, w))$ induces an isomorphism between the semilattices of all vertex-generated ideals of $L_{\K}(E, w)$ and $L_{\K}(F, w)$, which provides further evidence in support of Conjecture~\ref{conjalg}(2) for vertex-weighted graphs. 

To do so, we begin by defining the graph monoid \(M_{(E,w)}\) and the talented monoid \(T_{(E,w)}\) associated with a vertex-weighted graph \((E,w)\). We show that \(T_{(E,w)}\) is a cancellative monoid and that its Grothendieck group completion is precisely the graded Grothendieck group \(K_0^{\gr}(L_\K(E,w))\). The cancellativity of \(T_{(E,w)}\) implies that it coincides with the positive cone of \(K_0^{\gr}(L_\K(E,w))\), allowing one to work directly with the monoid rather than with the \(K_0^{\gr}\)-group. Working with the talented monoid has the advantage of being independent of the algebra itself while still providing detailed information about its elements (such as minimal elements, atoms, and related invariants) and, consequently, about the geometry of the underlying graph, including the number of cycles and their lengths. Moreover, we use Bergman machinery \cite{bergman74} and its graded version developed in \cite{HazLiP} to establish Conjecture~\ref{conjalg}(1) for vertex-weighted graphs.

The paper is organised as follows. In Section \ref{talentedmon}, we prove that the graph monoid associated with an acyclic vertex-weighted graph is cancellative (Theorem \ref{prop:cancel-gmonoid}). Consequently, the talented monoid associated with a vertex-weighted graph is cancellative (Proposition~\ref{goodviasm}(2)). We also give nice connections between the graph monoid $M_{(E,w)}$ of a vertex-weighted graph $(E, w)$ and the talented monoid $T_{(E,w)}$ (Proposition \ref{goodviasm}). Moreover, we show that, in general, the talented monoid associated with a vertex-weighted graph cannot be realized as the talented monoid of an ordinary graph (Example \ref{ggfear}). In Section \ref{distin}, based on Section \ref{talentedmon}, we demonstrate that the positive cone of  the graded Grothendieck group $K^{\gr}_0(L_{\K}(E, w))$ of the Leavitt path algebra of a vertex-weighted graph $(E, w)$ is exactly the talented monoid of $(E, w)$ (Corollary \ref{cor:smashprod}), and show that the graded Grothendieck group $K^{\gr}_0$ distinguishes the class of vertex-weighted Leavitt path algebras from the unweighted class (Theorem \ref{mainresult-sec3}). Moreover, we prove that every order preserving $\mathbb{Z}$-module isomorphism between $K_0^{\gr}(L_{\K}(E, w))$ and $K_0^{\gr}(L_{\K}(F, w))$ provides an isomorphism between the semilattices of all vertex-generated ideals of $L_{\K}(E, w)$ and $L_{\K}(F, w)$ (Theorem \ref{mainesult-sec3.2}). In Section \ref{lifting}, we prove that Conjecture~\ref{conjalg}(1) holds for vertex-weighted graphs (Theorems \ref{weightedlift} and \ref{arnonevas}) and show that the graded Grothendieck group $K_0^{\gr}$ classifies the Leavitt algebras $L_\K(n,n+k)$ (Theorem \ref{thm:classify-gLAs}).

\section{Weighted graphs and graph monoids}\label{talentedmon}
The main goal of this section is to establish the graph monoid associated with an acyclic vertex-weighted graph is cancellative (Theorem \ref{prop:cancel-gmonoid}). Also, we completely describe the structure of all order ideals of the graph monoid associated with a vertex-weighted graph (Theorem \ref{sat-hered-ord-id-lat}) via the lattice of all saturated hereditary subsets. Moreover, we give several connections between the graph monoid $M_{(E,w)}$ associated with a vertex-weighted graph $(E, w)$ and the talented monoid $T_{(E,w)}$ (Proposition~\ref{goodviasm}); in particular, we provide a lattice isomorphism between the lattice of all order ideals of $M_{(E,w)}$ and the lattice of all $\mathbb Z$-order ideals of $T_{(E,w)}$.
\medskip

We start by setting some terminology and notation about graphs.  A \emph{(directed) graph} is a quadruple $E=(E^0,E^1,s,r)$, where $E^0$ and $E^1$ are sets and $s,r:E^1\rightarrow E^0$ are maps.   Throughout, ``graph" will always mean ``directed graph".  The elements of $E^0$ are called \emph{vertices} and the elements of $E^1$ \emph{edges}.   (We allow the empty set to be viewed as a graph with $E^0 = E^1 = \varnothing.$)   If $e$ is an edge, then $s(e)$ is called its \emph{source} and $r(e)$ its \emph{range}. If $v$ is a vertex and $e$ an edge, we say that $v$ {\it emits} $e$ if $s(e)=v$, and $v$ {\it receives} $e$ if $r(e)=v$.   An edge $e$ is called a {\it loop at} $v$ in case $s(e) = v = r(e)$.  
A vertex is called a
% \emph{source} if it receives no edges,
 {\it sink} if it emits no edges.
A vertex is called \emph{regular} if it is not a sink and does not emit infinitely many edges.  The subset of $E^0$ consisting of all the regular vertices is denoted by $E^0_{\reg}$. Similarly, the subset of $E^0$ consisting of all the sinks  is denoted by $E^0_{\ssink}$.

 A graph is called \emph{row-finite} if any vertex emits a finite number (possibly zero) of edges. The graph $E$ is called  \emph{finite} if $E^0$ and $E^1$ are finite sets.   In this paper we assume throughout that all graphs are  row-finite,  and thus $E^0_{\reg}$ consists of all vertices which are not sinks.

A  {\it path}  $p$ in $E$ is
a sequence $p=e_{1}e_{2}\cdots e_{n}$ of edges in $E$ such that
$r(e_{i})=s(e_{i+1})$ for $1\leq i\leq n-1$. We define $s(p) = s(e_{1})$, and $r(p) =r(e_{n})$. 
 By definition, the \emph{length} $|p|$ of $p$ is $n$. We assign the length zero to vertices.  A \emph{closed path} (based at $v$) is a  path $p$ such that $s(p)=r(p)=v$. A \emph{cycle} (based at $v$) is a closed path $p=e_1 e_2 \cdots e_n$ based at $v$ such that $s(e_i)\neq s(e_j)$ for any $i\neq j$.  We say a vertex $v\in E^0$ has \emph{no bifurcation} if $| s^{-1}(u)| \leq 1$ for all $u\in T(v)$, where $T(v)$ is the usual tree of the vertex $v$. 

A \emph{weighted graph} is a pair $(E,w)$, where $E$ is a graph and $w:E^1\rightarrow \mathbb N^+$ is a map. If $e\in E^1$, then $w(e)$ is called the \emph{weight} of $e$. 
For each regular vertex $v$ in a weighted graph $(E,w)$ we set $w(v):=\max\{w(e)\mid e\in s^{-1}(v)\}$.  This gives  a map (called $w$ again) $w:E_{\reg}^0\rightarrow \mathbb N^+$. 
A weighted graph is called a \emph{vertex-weighted graph} if, for each regular vertex $v$, all edges emitted from $v$ have the same weight, 
i.e., if $w(e) = w(e')$ for all $e,e' \in s^{-1}(v)$.  In this case coincides with the weight of any edge emitting from $v$. A weighted graph $(E, w)$ is called {\it acyclic} if the graph $E$ has no cycles.

Let $(E,w)$ be a weighted graph. An edge $e\in E^{1}$ is called {\it unweighted} if $w(e)=1$ and {\it weighted} if $w(e)>1$. The subset of $E^1$ consisting of all unweighted edges is denoted by $E_{uw}^1$ and the subset consisting of all weighted edges by $E_{w}^1$.% The weighted subgraph $(E_{T(E^{0}_w)},w_{T(E^{0}_w)})$ of $(E,w)$ defined by the hereditary subset $T(E^{0}_w)\subseteq E^0$ is called the {\it weighted part of $(E,w)$}.

%We are in a postion to define the commutative monoids we assigned to the weighted graphs. 

The notion of the monoid associated to a weighted graph was introduced by the second author in \cite{H} and later generalised to arbitrary weighted graphs by R. Preusser in \cite{P}.

\begin{deff}[{\cite[Definition 11]{P}}]\label{wgmonoid}
Let $(E,w)$ be a row-finite vertex-weighted graph. 
We define the  {\it weighted graph monoid} of $(E,w)$ to be 
\begin{equation}\label{graphmon}
M_{(E,w)} \  \ := \mathbb{F}_E  \  \Big / \ \Big \langle w(v) v= \sum_{e \in s^{-1}(v)} r(e) \mid v\in E_{\reg}^0\Big \rangle,
\end{equation} where $\mathbb{F}_E$ is  the free commutative monoid on the set $E^0$ of vertices of $E$.    
\end{deff}

We note that for an unweighted graph $E$ (that is, when the weight function is constantly equal to $1$), the graph monoid~(\ref{graphmon}) reduces to the commutative monoid $M_E$ originally defined in \cite{amp}. Moreover, there exist weighted graph monoids that cannot be realised as monoids associated to unweighted graphs (see, for example, \cite[Example 3.10]{haznam}).

Let $(E, \omega)$ be a vertex weighted graph. There is an explicit description of the congruence on the free commutative monoid $\mathbb{F}_E$ given by the defining relations of $M_{(E, \omega)}$ in Definition \ref{wgmonoid}, as follows. For $v\in E^0$, define the ``$\rr$-transform on $\mathbb{F}_E$" by setting  
\begin{equation}\label{shtrans}
\rr(\omega(v) v) \ : = \ \sum_{e\in s^{-1}(v)}r(e) \text{ for al } v\in E^0_{\reg}.
\end{equation}
The nonzero elements of $\mathbb F_E$ can be written uniquely up to permutation as $\sum_{i=1}^{n}k_iv_{i}$, where $v_{i}$ are distinct vertices and $k_i\in \mathbb N^+$. Define a binary relation
$\rightarrow_{1}$ on $\mathbb F_E$ by 
\begin{equation}\label{hfgtrgt655}
\sum_{i=1}^{n}k_iv_{i} \ \longrightarrow_{1} \  \Big( \sum_{i\neq
j}k_iv_{i} \Big) +(k_j-w(v_j))v_j+ \rr(w(v_j)v_j),
\end{equation}
whenever $j\in \{1, \cdots, n\}$ and 
$k_j\geq w(v_{j})$.    Let $\rightarrow$ be the transitive and reflexive closure of $\rightarrow_{1}$
on $\mathbb F_E$.  Namely 
\begin{equation}\label{hfgtrgt6551}
a\rightarrow b    \ \  \ \text{ if }  a=b, \ \mbox{or}  \ a=a_0 \rightarrow_1 a_1 \rightarrow_1 \dots \rightarrow_1 a_k=b.
\end{equation}

\noindent
Finally, let   $\sim$   be the congruence on $\mathbb F_E$ generated by the relation $\rightarrow$.   That is, $a\sim b$ in case  there is a string $a=a_0, a_1,\dots, a_n=b$ in $\mathbb F_E$ such that $a_i\rightarrow_1 a_{i+1}$ or $a_{i+1}\rightarrow_1 a_{i}$ for each $0\leq i \leq n-1$.   Then 
$$M_{(E,w)}=\mathbb F_E/\sim.$$
\noindent
To avoid cumbersome equivalence class notation,  is standard (but not technically correct) to denote the elements of $M_{(E,w)}$ and the elements of $\mathbb F_E$ using the same symbols.  For instance, we  will sometimes write $a=b$ in $M_{(E,w)}$ for  elements $a,b \in \mathbb F_E$.

The relation~\ref{hfgtrgt6551} may be interpreted as follows. Let 
\begin{center}
$x=\sum_i k_ix_i \in \mathbb{F}_E$,\quad\quad $x_i\in E^0$.    
\end{center}
If $k_i\geq w(x_i)$, then we may ``let the vertex $x_i$ flow'' by replacing the term $w(x_i)x_i$ with the sum of the vertices adjacent to $x_i$. More precisely, we obtain that

$$y_1=\left(\sum_{j\neq i}k_jx_j\right)+(k_i-w(x_i))x_i+\sum_{e\in s^{-1}(x_i)}r(e),$$ and write \(x\to y_1\). 
Observe that this operation is permitted only when the summand $k_i x_i$ appears in the expression for $x$ and satisfies $k_i\geq w(x_i)$.

Repeating this procedure and allowing a vertex of $y_1$ to flow, we obtain another element $y_2\in \mathbb F_E$ such that $y_1\to y_2$. In general, the relation $\to_1$ is applied successively to vertices appearing in the decomposition of elements of $\mathbb F_E$. By definition of the relation $\to$, every element $y\in \mathbb F_E$ satisfying $x\to y$ is obtained from $x$ through a finite sequence of such vertex flows.

Next we recall the  ``confluence"  property of the monoid $M_{(E,w)}$, established by Abrams and the second author in \cite[Lemma 3.1]{GeneRooz}. Informally, this property asserts that whenever two elements $a,b\in \mathbb F_E$ represent the same element in the quotient monoid $$M_{(E,w)}=\mathbb F_E/\sim,$$ there exists a common element of $\mathbb F_E$ to which both $a$ and $b$ flow.   

\begin{lemma}[The Confluence Lemma]\label{aralem6}
Let $(E,w)$ be a vertex-weighted graph and $M_{(E,w)}$ its associated monoid. For $a, b \in \mathbb F_E$, we have   $a=b$ in $M_{(E,w)}$  (i.e., $a\sim b$ in $\mathbb F_E$)  if and only if there exists  $c \in \mathbb F_E$ such that  $a \rightarrow c$ and $b\rightarrow c$. 
\end{lemma}

Consequently, we obtain the following result, which follows immediately from \cite[Theorem 5.21]{H}, where Bergman algebras are used in the proof.

\begin{cor}\label{Conical}
For any vertex weighted graph $(E, \omega)$, the monoid $M_{(E, \omega)}$ is conical, i.e., for all $x, y\in M_{(E, \omega)}$, $x+y=0$ implies $x =y=0.$	
\end{cor}
\begin{proof}
Let $x$ and $y$ be two elements in $M_{(E, \omega)}$ such that $x + y =0$. Then, by The Confluence Lemma, $ x + y \rightarrow 0$ in $\mathbb{F}_E$.  Suppose that
there is a vertex $w\in E^0$ appearing in the presentation of $x$. We note that any possible transformations of $w$ or its multiple would give a vertex and subsequently any further transformations always contain a vertex. Thus $x + y$ can not be transformed to $0$,  a contradiction.
\end{proof}

The next part of this section is to show that the weighted graph monoid associated to an  acyclic vertex-weighted graph is cancellative. To do so, we first establish several useful results. Recall that for a graph $E=(E^0,E^1,s,r)$, there exists a preorder ``$\geq$'' on $E^0$ defined by $u\ge v$
whenever there exists a path $p$ in $E$ such that $s(p)=u$ and $r(p)=v$. We write $u > v$ if $u\ge v$ and $u\neq v$. We also note that if $E$ is acyclic, then the preorder is exactly a partial order. 
%Given a set of vertices $S \subseteq E^0$, we define $$\text{Maximals}_{\ge}(S):= \{ u \in E^0 \mid \nexists\ x\in S \text{ such that } x \ge u \}.$$

\begin{lemma}\label{lm:redu}
Let $(E,w)$ be a finite acyclic vertex-weighted graph, and let $a$ be an arbitrary element of $M_{(E, w)}$. Then $a$ can be  written in the form 
\begin{equation}\label{lastt}
 a = \sum_{v\in E^0} k_vv,   
\end{equation}
 where $0\le k_v < w(v)$ for every $v \in E^0_{\reg}$.
\end{lemma}

\begin{proof}
If $a =0$, then the conclusion is obvious. Hence, we can assume that $a \neq 0$. Then, there exists a nonempty subset $X\subseteq E^0$ such that
\begin{equation*}a = \sum_{v\in X}m_vv,  \end{equation*}where $m_v\in \mathbb{N}\setminus\{0\}$. Let $$H_X := \{v\in X \mid v\in E^0_{\reg} \text{ and } m_v \ge w(v)\}.$$
If $H_X = \emptyset$, then $a$ is already in the form of~(\ref{lastt}).  Otherwise, we denote by $H_X^{\max}$ the set of all maximal elements of $H_X$ with respect to the partial ordering $\geq$. This set is not empty as $E$ is a finite acyclic graph. For $u \in H_X^{\max}$, we write $m_u = q_uw(u) + r_u$, where $0\le r_u < w(u)$. Since $w(u)u = \sum_{e\in s^{-1}(u)}r(e)$, it follows that $$m_uu = r_uu + \sum_{e\in s^{-1}(u)}q_u r(e).$$ Therefore, we obtain that 
$$a = r_uu + \sum_{v\in H^{\max}_X\setminus\{u\}}m_vv + \sum_{e\in s^{-1}(u)}q_u r(e) + \sum_{v\in X\setminus H^{\max}_X}m_vv.$$
Continuing this same process now for $v\in H^{\max}_X\setminus \{u\}$, we see that after $|H^{\max}_X|$ steps we obtain that
%$$a = \sum_{u\in H^{\max}_X}r_uu + \sum_{u\in H^{\max}_X}\sum_{e\in s^{-1}(u)}q_u r(e) + \sum_{v\in X\setminus H^{\max}_X}m_vv.$$
\begin{align*}
 	a&=\sum_{u\in H^{\max}_X}r_uu + \sum_{u\in H^{\max}_X}\sum_{e\in s^{-1}(u)}q_u r(e) + \sum_{v\in X\setminus H^{\max}_X}m_vv\\&=\sum_{u\in H^{\max}_X}r_uu + \sum_{v\in r(s^{-1}(H^{\max}_X))) \cup (X\setminus H^{\max}_X)}m'_vv.
 \end{align*}
Let 
\begin{center}
$X_1 = X \cup r(s^{-1}(H^{\max}_X))$ and $H_{X_1}:= \{v\in X_1 \mid v\in E^0_{\reg} \text{ and } m'_v \ge w(v)\}$.    
\end{center}
If $H_{X_1} = \emptyset$, then the conclusion is obvious.
Otherwise, we denote by $H^{\max}_{X_1}$ the set of all maximal elements of $H_{X_1}$. Since $E$ is acyclic, it follows that $H^{\max}_{X} \cap H^{\max}_{X_1} = \varnothing$
and for each $v\in H^{\max}_{X_1}$, there exists a vertex $u\in H^{\max}_{X}$ such that $u > v$. Continuing the above same process now for $v\in H^{\max}_{X_1}$, we see that after $|H^{\max}_{X_1}|$ steps we obtain that
%$$a = \sum_{u\in H^{\max}_X}r_uu + \sum_{u\in H^{\max}_X}\sum_{e\in s^{-1}(u)}q_u r(e) + \sum_{v\in X\setminus H^{\max}_X}m_vv.$$
\begin{align*}
 	a&=\sum_{u\in H^{\max}_X}r_uu + \sum_{x\in H^{\max}_{X_1}}r_xx+\sum_{v\in r(s^{-1}(H^{\max}_{X_1})) \cup (X_1\setminus (H^{\max}_X \cup H^{\max}_{X_1}))}m''_vv,
 \end{align*}
where $r_u < w(u)$ for all $u\in H^{\max}_{X}$ and 
$r_x < w(x)$ for all $x\in H^{\max}_{X_1}$.

Continuing this same process now on the set $X_2:= X_1 \cup r(s^{-1}(H^{\max}_{X_1}))$, we see that after $n_0$ steps we obtain that $H_{X_{n_0}} = \varnothing$, thus finishing the proof.
\end{proof}

\begin{lemma}\label{lm:redu-sink}
Let $(E,w)$ be a finite acyclic vertex-weighted graph, and let $a$ be a nonzero element of $\mathbb{F}_E$. Then there exist a positive integer $n$ and a nonzero element $$a' = \sum_{v\in E^0_{\ssink}}n_vv\in \mathbb{F}_E$$ such that $na\rightarrow a'$.
\end{lemma}
\begin{proof}
Since $a\in \mathbb{F}_E\setminus \{0\}$, there exists a nonempty subset $X\subseteq E^0$ such that
\begin{equation*}a = \sum_{v\in X}m_vv,  \end{equation*}where $m_v\in \mathbb{N}\setminus\{0\}$. We now define $H_0 = X$ and for $n\ge 1$ we define inductively $$H_n = r(s^{-1}(H_{n-1})) \cup (H_{n-1} \cap E^0_{\ssink}).$$
Since $E$ is finite and acyclic, there exists a natural number $n_0$ such that $H_{n_0}$ is a nonempty subset of $E^0_{\ssink}$. Among these numbers, we may choose $n_0$ to be minimal. For $z \in H_1\setminus E^0_{\ssink}$, there exist a vertex $u\in X$ and an edge $f\in s^{-1}(u)$ such that $z = r(f).$ Let $k_u$ be the smallest positive integer such that $k_um_u = l_uw(u)$ for some $l_u\in \mathbb{N}$. We then have 
\begin{equation*}k_ua = k_um_u u + \sum_{v\in X\setminus\{u\}}k_um_vv  \end{equation*}
and  $$w(u)u = \sum_{e\in s^{-1}(u)}r(e)= z + \sum_{e\in s^{-1}(u)\setminus\{f\}}r(e),$$ where $r(e)\in H_1$ for all $e\in s^{-1}(u)\setminus\{f\}$. Therefore, we obtain that
$$k_ua = \sum_{v\in X\setminus\{u\}}k_um_vv + l_u z + \sum_{e\in s^{-1}(u)\setminus\{f\}}l_ur(e).$$  Continuing this same process now for $v\in H_1\setminus  (E^0_{\ssink} \cup \{z\})$, we see that after at most $|H_1\setminus  E^0_{\ssink}|$ steps we obtain that
$$n'a = \sum_{v\in H_1}m'_v v,$$ where $n'\ge 1$ and each $m'_v \ge 1$. Continuing this same process now on the set $H_1$, we see that after $n_0$ steps we obtain that
$$na = \sum_{v\in H_{n_0}}n_v v,$$  where $n\ge 1$ and each $n_v \ge 1$, thus finishing the proof.
\end{proof}

Let \((E,w)\) be a weighted graph. A {\it weighted subgraph} $X$ of $(E, w_E)$ is the weighted graph $$X = (X^0, X^1, r_X, s_X, w_X)$$ such that $X^0\subseteq E^0$, $X^1\subseteq E^1$, and $r_X, s_X$ are respectively the restrictions of $s_E, r_E$ on $X^1$ and 
\begin{center}
$w_X(e) = w_E(e)$ for all $e\in X^1$.   
\end{center}
A weighted subgraph $X$ is called {\it complete} if $$s^{-1}_X(v) = s^{-1}_E(v)$$ for all $v\in X^0_{\reg}$. 
By \cite[Lemma 5.19]{H}, the weighted graph $E$ can be expressed as the direct limit $$(E,w)=\varinjlim (E_i,w_i),$$
where $\{(E_i,w_i)\}$ is a directed system of finite  complete weighted subgraphs of $(E,w)$. 

We are now in a position to state the first main result of this section.

\begin{thm}\label{prop:cancel-gmonoid}
For every acyclic vertex-weighted graph $(E,w)$, the monoid $M_{(E,w)}$ is cancellative.    
\end{thm}
\begin{proof}
Let $(E, w)$ be an arbitrary acyclic vertex-weighted graph. By \cite[Lemma 5.19 (1)]{H}, it follows that $$(E,w)=\varinjlim (E_i,w_i),$$ where each $(E_i, w_i)$ is a finite complete weighted subgraphs of \((E,w)\).  Since $(E, w)$ is an acyclic vertex-weighted graph, each $(E_i,w_i)$ is also an acyclic vertex-weighted graph. By \cite[Lemma 5.19 (2) and Theorem 5.21]{H}, the monoid $M_{(E, w)}$ can be realised as the directed limit $$M_{(E,w)}=\varinjlim M_{(E_i,w_i)}.$$
From this observation, it suffices to prove that $M_{(E, w)}$ is cancellative in the case where $(E, w)$ is a finite acyclic vertex-weighted graph.
    
Assume that $(E,w)$ is a finite acyclic vertex-weighted graph. Let $a,b,c$ be elements of $\mathbb F_E$ such that $a+c = b+c$ in $M_{(E,w)}$. Equivalently, $a+ nc = b+ nc$ in $M_{(E,w)}$ for all $n\ge 1$.
We claim that $a = b$ in $M_{(E,w)}$. The conclusion is immediate when $c=0$. Thus, we may assume that $c\neq 0$. Then, by Lemma \ref{lm:redu-sink}, there exist a positive integer $n$ and a nonzero element $$c' = \sum_{v\in E^0_{\ssink}}k_v v\in \mathbb{F}_E$$ such that $nc \rightarrow c'.$ Consequently, we have 
$nc = c'$ in $M_{(E, w)}.$    
By Lemma \ref{lm:redu}, there exist elements 
$a' = \sum_{v\in E^0}n_v v$ and $b' = \sum_{v\in E^0}v_v v\in \mathbb{F}_E$  
such that 
\begin{center}
$a = a'$   and $b = b'$ in $M_{(E, w)},$  
\end{center}
where $m_v < w(v)$ and $n_v < w(v)$ for all $v\in E^0_{\reg}$. From these observations, it follows that 
\begin{center}
$a' + c' = b' + c'$ in $M_{(E,w)}$.    
\end{center}
By the Confluence Lemma~\ref{aralem6}, there exits an element $x \in \mathbb{F}_E$ such that  $a' + c' \rightarrow x$ and $b' + c' \rightarrow x$.    
By the construction of $a'$, $b'$ and $c'$, neither $a'+c'$ nor $b'+c'$ admits any further reductions. Hence, 
$a' + c' = x = b' + c'$ in $\mathbb{F}_E$. This implies that $a' = b'$ in $\mathbb{F}_E$. Consequently, we have
$a = b$ in $M_{(E, w)}$.    
This finishes the proof.
\end{proof}

In the following examples, we present a vertex-weighted graph whose associated graph monoid is cancellative, and another vertex-weighted graph whose graph monoid is not cancellative.

\begin{example}\label{cycliccancellative}
(1) Let $(E,w)$ be the following vertex-weighted graph
    
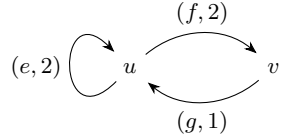
\[
\begin{tikzpicture}[
    >=Stealth,
    shorten >=1pt,
    every node/.style={font=\small},
    scale=0.95,
    transform shape
]

\node (u) at (0,0) {$u$};
\node (v) at (2,0) {$v$};

% edges between u and u
\draw[->]
    (u) edge[
        in=135,
        out=225,
        looseness=8
    ]
    node[left] {$(e,2)$}
    (u);

% edges between u and v
\draw[->] 
    (u) to[bend left=40] 
    node[midway, above] {$(f,2)$} 
    (v);

\draw[->] 
    (v) to[bend left=40] 
    node[midway, below] {$(g,1)$} 
    (u);
    
\end{tikzpicture}
\]
where $w(e) =2= w(f)$ and $w(g) =1$. We then have $$M_{(E,w)} = \mathbb{F}_{\{u, v\}} \big/ \big\langle 2u=u+v, ~ v=u \big\rangle\cong \mathbb{F}_{\{u, v\}} \big/ \big\langle  v=u \big\rangle \cong \mathbb{N}.$$ 
This implies that $M_{(E,w)}$ is cancellative.

(2) Let $(E,w)$ be the following vertex-weighted graph 
    
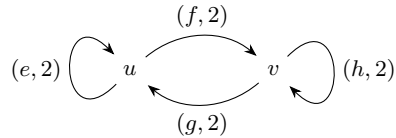
\[
\begin{tikzpicture}[
    >=Stealth,
    shorten >=1pt,
    every node/.style={font=\small},
    scale=0.95,
    transform shape
]

\node (u) at (0,0) {$u$};
\node (v) at (2,0) {$v$};

% edges between u and u
\draw[->]
    (u) edge[
        in=135,
        out=225,
        looseness=8
    ]
    node[left] {$(e,2)$}
    (u);

% edges between u and v
\draw[->] 
    (u) to[bend left=40] 
    node[midway, above] {$(f,2)$} 
    (v);

\draw[->] 
    (v) to[bend left=40] 
    node[midway, below] {$(g,2)$} 
    (u);

% edges between v and v
\draw[->]
    (v) edge[
        in=-45,
        out=45,
        looseness=8
    ]
    node[right] {$(h,2)$}
    (v);
    
\end{tikzpicture}
\] where $w(e) = w(f) = w(g) = w(h) = 2$. We then have 
$$M_{(E,w)} = \mathbb{F}_{\{u, v\}} \big/ \big\langle 2u=u+v, ~ 2v=u +v \big\rangle. $$
Assume that $M_{(E,w)}$ is cancellative. Then the relation $2u=u+v$ would imply $u=v$ in $M_{(E,w)}$. 
By The Confluence Lemma, there exits an element $x \in \mathbb{F}_E$ such that 
\begin{center}
$u \rightarrow x$ and $v \rightarrow x$.    
\end{center}
However,  neither $u$ nor $v$ admits any further reductions, and so $u = x = v$ in $\mathbb{F}_E$, a contradiction.
Hence, the monoid $M_{(E,w)}$ is not cancellative.
\end{example}

We note that for every (unweighted) graph $E$, the graph monoid $M_E$ is cancellative if and only if no cycle in $E$ has exits (see, e.g., \cite[Theorem 4.2]{GeneNamPhuc}). However, Example \ref{cycliccancellative}(1) shows that this characterization does not, in general, extend to vertex-weighted graphs. In light of this observation, it is natural to ask the following question.

\begin{ques}
Can one characterise all vertex-weighted graphs whose graph monoids are cancellative?
\end{ques}

%in Example \ref{cycliccancellative}(1), we used the fact that the monoid \(M_{(E,w)}\) is free to show cancellativity. In the case of finite unweighted graphs, a similar phenomenon occurs; in fact, the graph monoid of a finite unweighted graph is cancellative if and only if it is free (in the proof of \cite[Lemma 5.5]{ara-hazrat-li-sims}). One distinction between weighted and unweighted graph monoids is that a cancellative weighted graph monoid need not be free, as illustrated in the following example.\begin{example}
%    Let $(E,w)$ be the following vertex-weighted graph     \[
%\begin{tikzpicture}[    >=Stealth,
 %   shorten >=1pt,
%    every node/.style={font=\small},
%    scale=0.95,
%    transform shape]

%\node (u) at (0,0) {$u$};
%\node (v) at (2,0) {$v$};

% edges between u and v\draw[->] 
 %   (u) to[bend left=30] 
%    node[midway, above] {$(e,3)$}     (v);
%\draw[->] 
%    (u) to[bend right=30] 
%    node[midway, below] {$(f,3)$} 
%    (v); 
%\end{tikzpicture}
%\]
%where $w(e)=w(f)=3.$ We then have 
%\[
%M_{(E,w)} = \F_{\{ u,v \}} \big/ \big\langle  3u=2v
%\big\rangle
%\]
%Since $E$ is acyclic, by Proposition \ref{prop:cancel-gmonoid}, $M_{(E,w)}$ is cancellative. On the other hand, free commutative monoids admit no {nontrivial relations among their presentation}. %needs improvementSince $3u=2v$ in \(M_{(E,w)}\), it follows that \(M_{(E,w)}\) is not free.\end{example}

In the remainder of this section, we describe the structure of weighted graph monoids $M_{(E, \omega)}$ in terms of saturated hereditary subsets of $E$. To do so, we need to recall some useful notions. 

\begin{deff}[{\cite[Definition 2.7]{haznam}}]\label{sathersubsets}
Let $(E, \omega)$ be a  vertex-weighted graph and $H\subseteq E^0$. 

(1) The set $H$ is called {\it hereditary} if for any $e\in E^1$, $s(e)\in H$ implies $r(e)\in H$. 

(2) The set $H$ is {\it saturated} if  whenever $v\in E_{\reg}^0$ with the property that $r(s^{-1}(v))\subseteq H$, then $v\in H$.

(3) The {\it hereditary saturated closure} of $H$, denoted by $\overline{H}$, is the
smallest hereditary and saturated subset of $E^0$ containing $H$.	
\end{deff} 
We denote by $\mathcal{H}_{(E,\omega)}$ the set of all  saturated hereditary subsets of $E^0$.

\begin{example}
 Let $(E, \omega)$ be the following vertex weighted graph
% $$\xymatrix{\bullet^{w}\ar[r]^{e}&
% 	\bullet^{v} \ar@(ul,ur)^f\ar[r]^g& \bullet^{s}&\bullet^{u}\ar[l]_{h}}$$
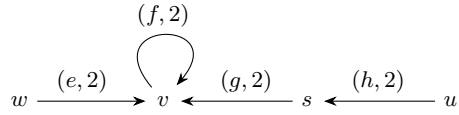
\[ \begin{tikzpicture}[
    >=Stealth,
    shorten >=1pt,
    every node/.style={font=\small},
    scale=0.95,
    transform shape
]

\node (w) at (-2,0) {$w$};
\node (v) at (0,0) {$v$};
\node (s) at (2,0) {$s$};
\node (u) at (4,0) {$u$};

% edges 
\draw[->]
    (w) -- node[above, transform canvas={xshift=-4pt}] {$(e,2)$}
    (v);
    
\draw[->]
    (v) edge[
        in=50,
        out=130,
        looseness=10
    ]
    node[above] {$(f,2)$}
    (v);

\draw[->]
    (s) -- node[above, transform canvas={xshift=4pt}] {$(g,2)$}
    (v);

\draw[->]
    (u) -- node[above] {$(h,2)$}
    (s);
\end{tikzpicture}
\]

	where all edges have weight $2$.	
Then we have $$\mathcal{H}_{(E,\omega)} =
\big \{\varnothing, \{s, u\}, E^0\big\}.$$ 
\end{example}

Let $(E, \omega)$ be a  vertex-weighted graph and $H\in \mathcal{H}_{(E,\omega)}$. We denote by $(E_H, \omega_r)$ the {\it restriction weighted graph}
\begin{center}
		$E_H^0 := H$, \quad\quad $E_H^1 := \{e\in E^1\mid s(e)\in H\}$
\end{center}	
and the source, range and weight function in $E_H$ are exactly the source, range and weight function in $E$ restricted to $H$.

The following proposition provides that every nonempty saturated hereditary subset $H$ of a vertex-weighted graph $E$ induces a weighted graph submonoid $ M_{(E_H, \omega_r)}$ of $ M_{(E, \omega)}$.

\begin{prop}\label{sandsubmon}
Let $(E, \omega)$ be a  vertex-weighted graph and $H\in \mathcal{H}_{(E,\omega)}$. Then, the map
$\psi_H: M_{(E_H, \omega_r)}\rightarrow M_{(E, \omega)}$, defined by $v \mapsto v$, is an injective homomorphism of monoids. Consequently,
 $ M_{(E_H, \omega_r)}$ is a submonoid of $M_{(E, \omega)}$.	
\end{prop}
\begin{proof} 
Clearly the map $\psi_H: M_{(E_H, \omega_r)}\rightarrow M_{(E, \omega)}, v\mapsto v$ is well-defined. 
Let $a, b\in  M_{(E_H, \omega_r)}$ such that $a = b$ in $ M_{(E, \omega)}$. By the Confluence
Lemma~\ref{aralem6}, there exists an element $c\in \mathbb{F}_E$ such that $a \rightarrow c$ and $b\rightarrow c$. Since $H$ is hereditary, all the transformations occur in $E_H$, and so $c\in \mathbb{F}_{E_H}$. This implies that $a = b$ in $ M_{(E_H, \omega_r)}$. Hence, $\psi_H$ is  an injective homomorphism of monoids, thus finishing the proof.
\end{proof}	

Recall (see \cite[Page 131]{lpabook}) that an {\it order-ideal} of a commutative monoid $M$ is a submonoid $I$ of $M$ such that, for any $x, y\in M$, if $x+y \in I$ then $x, y\in I$. An order-ideal may also be described as a submonoid $I$ of $M$ which is hereditary with respect to the canonical preorder $\le$ on $M$: $x\le y$ and $y\in I$ imply $x\in I$, where  the preorder $\le$ on $M$ is defined by
setting $x\le y$ if  $y = x + m$ for some $m\in M$. For each $X\subseteq M$, the order-ideal of $M$ generated by $X$ is the set \[\langle X \rangle :=\big\{m\in M\mid m\le \sum^n_{i=1}x_i,\ n\in \mathbb{N}^+,\ x_i\in X\big\}.\]
The set $\mathcal{L}(M)$ of order-ideals of $M$ forms a complete lattice such that the join of two elements $I, J$ is exactly the order-ideal of $M$ generated by $I +J$.

We are now in a position to establish the second main result of this section describing order-ideals of the graph monoid $ M_{(E, \omega)}$ of a vertex-weighted graph $(E, \omega)$, which is an analogue of the corresponding result for sandpile monoids established in \cite[Theorem 2.10]{haznam1}.

\begin{thm}\label{sat-hered-ord-id-lat}
For any  vertex weighted graph $(E, \omega)$,   the following statements hold:

$(1)$ $0$ and $M_{(E, \omega)}$ are two order-ideals of $M_{(E, \omega)}$;

$(2)$ For all $H\in \mathcal{H}_{(E,\omega)}$, $M_{(E_H, \omega_r)}$ is an order-ideal of $M_{(E, \omega)}$;

$(3)$ For every order-ideal $I$ of $M_{(E, \omega)}$, $I = M_{(E_H, \omega_r)}$, where $H= I\cap E^0\in \mathcal{H}_{(E,\omega)}$;

$(4)$ $\mathcal{H}_{(E,\omega)}\cong \mathcal{L}(M_{(E, \omega)})$ as lattices.
\end{thm}
\begin{proof}
(1) By Corollary \ref{Conical}, the zero ideal is an order-ideals of $M_{(E, \omega)}$. Also, $M_{(E, \omega)}$ is obviously an order-ideals of $M_{(E, \omega)}$.

(2) Let $H$ be an element in $\mathcal{H}_{(E,\omega)}$. By Proposition \ref{sandsubmon}, $M_{(E_H, \omega_r)}$ is a submonoid of $M_{(E, \omega)}$. Assume that $x$ any $y$ are elements in $M_{(E, \omega)}$ with $x\le y$ and $y\in M_{(E_H, \omega_r)}$. Then, $y = x +z$ in $M_{(E_H, \omega_r)}$ for some $z\in M_{(E, \omega)}$. By the Confluence Lemma, there exists an element $c\in \mathbb{F}_E$ such that $y \rightarrow c$ and $x+z\rightarrow c$. Since $y\in M_{(E_H, \omega_r)}$, $\text{supp}(y)\subseteq E^0_H$, and so $\text{supp}(c)\subseteq E^0_H$. Assume that $x\notin M_{(E_H, \omega_r)}$. Then, there exists a vertex $v\in \text{supp}(x)$ such that $v\notin E^0_H$. Since $H$ is saturated and hereditary, any possible transformations of $v$ or its multiple would give a vertex  which is not in $E_H$ and subsequently any further transformations always contain a vertex  which is not in $E_H$. This shows that $x + y$ cannot be transformed to $c$, a contradiction, and so $x\in M_{(E_H, \omega_r)}$. Thus $M_{(E_H, \omega_r)}$ is an order-ideal of $M_{(E, \omega)}$.

(3) Let $I$ be an order-ideal of $M_{(E, \omega)}$ and $H := I \cap E^0$. We claim that $H\in \mathcal{H}_{(E,\omega)}$. Indeed, let $e\in E^1$ with $v:=s(e)\in H$. We then have $v\in I$ and $\sum_{f\in s^{-1}(v)}r(f)=\omega(v) v \in I$, and so $r(f)\in I$ for all $f\in s^{-1}(v)$ (since $I$ is an order-ideal of $M_{(E, \omega)}$); in particular, $r(e)\in I$. Hence, $H$ is hereditary.
Now let $v$ be a regular vertex in $E$ with $r(s^{-1}(v))\subseteq I$. Then, $\omega(v)v=\sum_{f\in s^{-1}(v)}r(f) \in I$, and so $v\in I$. Therefore, $H$ is saturated, showing the claim.

We next prove that $I= M_{(E_H, \omega_r)}$. Since $H\subseteq I$ and $I$ is a submonoid of $M_{(E, \omega)}$, we have $M_{(E_H, \omega_r)}\subseteq I$. Let $x$ be a nonzero element in $I$. Since $I$ is an order-ideal of $M_{(E, \omega)}$, we must have $\text{supp}(x)\subseteq I$, and so $\text{supp}(x)\subseteq H$. This implies that $x\in M_{(E_H, \omega_r)}$, that means, $I\subseteq M_{(E_H, \omega_r)}$.

(4) Let $$\phi:  \mathcal{H}_{(E,\omega)}\longrightarrow \mathcal{L}(M_{(E, \omega)})$$ be the map defined by: $$H\longmapsto\phi(H)= M_{(E_H, \omega_r)},$$ and let $$\psi: \mathcal{L}(M_{(E, \omega)})\longrightarrow \mathcal{H}_{(E,\omega)}$$ be the map defined by: $$I\longmapsto \psi(I) = I\cap E^0.$$ It is obvious that both maps are order-preserving. Moreover, we have $\psi(\phi(H)) = \psi(M_{(E_H, \omega_r)}) = H$ for every $H\in \mathcal{H}_{(E,\omega)}$, and hence $\psi\phi = id_{\mathcal{H}_{(E,\omega)}}.$ Conversely, by item (3), we obtain that $\phi\psi = id_{\mathcal{L}(M_{(E, \omega)})}.$ Therefore, $\phi$ and $\psi$ are order-preserving mutually inverse maps. Consequently, we obtain that
$\mathcal{H}{(E,\omega)}\cong\mathcal{L}(M{(E,\omega)})$
as lattices, thus finishing the proof.
\end{proof}

For clarification, we consider the following example.

\begin{example}
 Let $(E, \omega)$ be the following vertex weighted graph
% $$\xymatrix{\bullet^{w}\ar[r]^{e}&
% 	\bullet^{v} \ar@(ul,ur)^f\ar[r]^g& \bullet^{s}&\bullet^{u}\ar[l]_{h}}$$
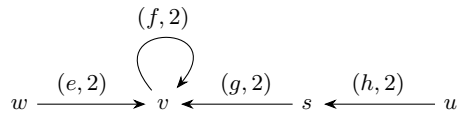
\[ \begin{tikzpicture}[
    >=Stealth,
    shorten >=1pt,
    every node/.style={font=\small},
    scale=0.95,
    transform shape
]

\node (w) at (-2,0) {$w$};
\node (v) at (0,0) {$v$};
\node (s) at (2,0) {$s$};
\node (u) at (4,0) {$u$};

% edges 
\draw[->]
    (w) -- node[above, transform canvas={xshift=-4pt}] {$(e,2)$}
    (v);
    
\draw[->]
    (v) edge[
        in=50,
        out=130,
        looseness=10
    ]
    node[above] {$(f,2)$}
    (v);

\draw[->]
    (s) -- node[above, transform canvas={xshift=4pt}] {$(g,2)$}
    (v);

\draw[->]
    (u) -- node[above] {$(h,2)$}
    (s);
\end{tikzpicture}
\]
where all edges have weight $2$.	
Then we have
\begin{center}
 $\mathcal{H}_{(E,\omega)} =\{\varnothing, \{s, u\}, E^0\}$ and  $\mathcal{L}(M_{(E, \omega)}) = \{0, \langle s, u \rangle,  M_{(E, \omega)}\}$,		
\end{center}
and so we have $\mathcal{H}_{(E,\omega)}\cong \mathcal{L}(M_{(E, \omega)})$ as lattices via the map: 
\begin{center}
$\varnothing \longmapsto 0$, $\{s, u\}\longmapsto \langle s, u \rangle$ and $E^0  \longmapsto M_{(E, \omega)}$.    
\end{center}
\end{example}

Next we define the ``talented'' version of the weighted graph monoids~(\ref{graphmon}). 

\begin{deff}\label{tal-wei-mon}
Let $(E,w)$ be a row-finite vertex-weighted graph. 
We define the  {\it talented monoid} of $(E,w)$ to be 
\begin{equation}\label{talmon}
T_{(E,w)} \  \ := \\ \Big \langle v(i) \mid v\in E^0,\ i \in \mathbb Z \Big \rangle  \  \Big / \ \Big \langle w(v) v(i)= \sum_{e \in s^{-1}(v)} r(e)(i+1) \, \big | \, v\in E_{\reg}^0 , i \in \mathbb{Z} \Big \rangle,
\end{equation}
where $\Big \langle v(i) \mid v\in E^0,\ i \in \mathbb Z \Big \rangle$ is the free commutative monoid on the set $\{v(i)\mid v\in E^0,\ i \in \mathbb{Z}\}$.

We should mention that for a directed row-finite  graph $E$, i.e., the weight is the constant $1$, the talented monoid~(\ref{talmon}) reduces to the commutative monoid $T_E$ originally defined in \cite{hazli}.  

A {\it $\Gamma$-monoid}, where $\Gamma$ is an abelian group, consists of a commutative monoid $M$ equipped with a group action of $\Gamma$ by monoid homomorphisms. The image of an element $m \in M$ under the action of a group element $\gamma \in \Gamma$ is denoted by ${ }^{\gamma} m$. A monoid homomorphism $\phi: M\longrightarrow M'$ is called a {\it $\Gamma$-homomorphism} if $\phi({ }^{\gamma} m) = { }^{\gamma} \phi(m)$ for all $m\in M$ and $\gamma \in \Gamma$.

There is a well-defined action of $\mathbb Z$ on the monoid $T_{(E,w)}$ defined on the generators by ${}^n v(i):=v(i+n)$, where $i, n \in \mathbb Z$, and extends to all of the monoid. 
Throughout we write $v:=v(0)$.  
We denote $1_E:=\sum_{v\in E^0}v$, a distinguished order unit element both in $M_{(E,w)}$ and $T_{(E,w)}$
\end{deff}

For clarification, we illustrate Definition \ref{tal-wei-mon} by presenting the following examples.
We first recall the notion of refinement for monoids. 
A commutative monoid $M$ is called a \emph{refinement monoid} if for any equality
\[
a_1 + \cdots + a_m \;=\; b_1 + \cdots + b_n,
\]
with $a_i, b_j \in M$, there exist elements $x_{ij} \in M$ such that
\[
a_i = \sum_{j=1}^n x_{ij} \quad \text{for all } i=1,\dots,m,
\]
and
\[
b_j = \sum_{i=1}^m x_{ij} \quad \text{for all } j=1,\dots,n.
\]

We shall show that only the talented monoids of unweighted Leavitt path algebras are refinement (Proposition~\ref{cortal}).

\begin{example}\label{hfgfyy}
Let $E$ be a graph with two loops of weight one. Then 
\[T_E=\frac{\langle v(i) \mid i\in \mathbb Z\rangle}{\langle v(i)=2v(i+1)\rangle}.\]
One can see that $T_E\cong \mathbb N[1/2]$, with the action ${}^1 n = 1/2 n$, where $n\in \mathbb N[1/2]$.

On the other hand, for the weighted graph $F$ consisting of two loops of weight 2, we have 
\[T_{(F,w)}=\frac{\langle v(i) \mid i\in \mathbb Z\rangle}{\langle 2v(i)=2v(i+1)\rangle}.\]
We obtain that $T_{(F,w)}$ is cancellative but not a refinement monoid. The cancellative property follows from Proposition~\ref{goodviasm}(2) below. Assume that $T_{(F,w)}$ is a refinement monoid. Since $2v(0) = 2v(1)$, there exist elements $a_{ij}\in T_{(F, w)}$, $1\le i, j\le 2$, such that 
$$v(0) = a_{11} + a_{12} =  a_{21} + a_{22} \text{ and } v(1) =  a_{11} + a_{21} =  a_{12} + a_{22}$$ in $T_{(F, w)}$. By Confluence Lemma \ref{aralem6}, we have $ a_{i1} + a_{i2}\rightarrow v(0)$ for all $i$. It follows that for each $i$, either $a_{i1}=0$ and $a_{i2}=v(0)$, or $a_{i2}=0$ and $a_{i1}=v(0)$. Without loss of generality, we can assume that $a_{11} = 0$ and $a_{12}=v(0)$. We then have $v(1) = v(0) + a_{22}$ in $T_{(F, w)}$. By Confluence Lemma \ref{aralem6}, both $v(1)$ and $v(0) + a_{22}$ flow to
the same element in the free monoid, and so $v(0) + a_{22}\rightarrow v(1)$. This implies that $a_{22} \rightarrow nv(0)$ for some $n\ge 1$. Therefore, $$v(1) =v(0) + a_{22} = (n+1)v(0) = 2v(1) + (n-1)v(0)$$ in $T_{(F, w)}$. Since $T_{(F, w)}$ is cancellative, $v(1) + (n-1)v(0) = 0$, and hence $v(1) =0$, a contradiction. Thus $T_{(F,w)}$ is not a refinement monoid.
\end{example}

\begin{example}\label{ggfear}
    
Consider the graph $E$, where all untagged edges are assigned weight one. 

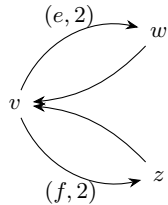
\[
\begin{tikzpicture}[
    >=Stealth,
    shorten >=1pt,
    every node/.style={font=\small},
    scale=0.95,
    transform shape
]

\node (v) at (0,0) {$v$};
\node (w) at (2,1) {$w$};
\node (z) at (2,-1) {$z$};

% edges between v and w
\draw[->] 
    (v) to[bend left=40] 
    node[midway, above] {$(e,2)$} 
    (w);

\draw[->] 
    (w) to[bend left=20] 
    (v);

% edges between v and z
\draw[->] 
    (v) to[bend right=40] 
    node[midway, below] {$(f,2)$} 
    (z);

\draw[->] 
    (z) to[bend right=20] 
    (v);

\end{tikzpicture}
\]
The commutative monoid $M_{(E,w)}$ is generated by symbols $v,w,z$ subject to the weighted graph relations $$2v=w+z,\ w=v,\ z=v.$$ One can see that the generators $w,z$ are redundant and the relation reduces to $2v=2v$ which is also redundant. Thus $M_{(E,w)}\cong \mathbb N$. We shall see that the weighted Leavitt path algebra $L_\K(E,w)$ is not isomorphic to an unweighted Leavitt path algebra (as the graph $(E,w)$ does not satisfy Preusser's conditions~\ref{defLPA}).  This example shows that the weighted graph monoid cannot distinguish the weighted Leavitt path algebra $L_K(E,w)$ from unweighted Leavitt path algebras. Although $L_K(E,w)$ is not isomorphic to any unweighted Leavitt path algebra, its graph monoid is $\mathbb{N}$, which also arises as the graph monoid of numerous unweighted Leavitt path algebras.

Now consider the talented monoid $T_{(E,w)}$. It is easy to see that 
\[T_{(E,w)}=\frac{\langle v(i) \mid i\in \mathbb Z\rangle}{\langle 2v(i)=2v(i+2)\rangle},\] which is not refinement (by Example~\ref{hfgfyy}). Consequently, $T_{(E,w)}$ cannot be isomorphic to a talented monoid of an ordinary graph (as they are refinement). We shall see in \S\ref{distin} that the 
talented monoid of  weighted graphs distinguishes the class of weighted Leavitt path algebras from the unweighted ones. 
\end{example}

In the remainder of this section, we heavily use the concept of the covering graph of a weighted graph. 

\begin{deff}\label{deff:covering}
The \emph{covering graph}  $(\overline E, \overline w)$  of 
a weighted graph $(E,w)$ (also denoted by  $E\times_1 \mathbb Z$) 
is a weighted graph defined  by
\begin{gather*}
    \overline E^0 = \big\{v^{(n)} \mid v \in E^0 \text{ and } n \in \Z \big\},\qquad
    \overline E^1 = \big\{e^{(n)} \mid e\in E^1 \text{ and } n\in \Z \big\},\\
    s(e^{(n)}) = s(e)^{(n)}, \qquad  r(e^{(n)}) = r(e)^{(n+1)} \qquad\text{ and } \qquad w(e^{(n)})=w(e).
\end{gather*}    
\end{deff}

For clarification, we illustrate Definition \ref{deff:covering} by presenting the following examples.

\begin{example}    
Consider the graph $E$ from Example~\ref{ggfear}, where all untagged edges are assigned weight one. 
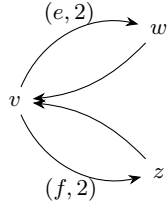
\[
\begin{tikzpicture}[
    >=Stealth,
    shorten >=1pt,
    every node/.style={font=\small},
    scale=0.95,
    transform shape
]

\node (v) at (0,0) {$v$};
\node (w) at (2,1) {$w$};
\node (z) at (2,-1) {$z$};

% edges between v and w
\draw[->] 
    (v) to[bend left=40] 
    node[midway, above] {$(e,2)$} 
    (w);

\draw[->] 
    (w) to[bend left=20] 
    (v);

% edges between v and z
\draw[->] 
    (v) to[bend right=40] 
    node[midway, below] {$(f,2)$} 
    (z);

\draw[->] 
    (z) to[bend right=20] 
    (v);

\end{tikzpicture}
\]
Then the covering graph $\overline E$ takes the following form, with edges emiting from $v^{(i)}$ has weight $2$ and others have weight $1$. 

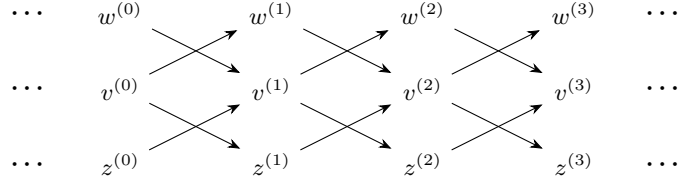
\[
\begin{tikzpicture}[
    >=Stealth,
    shorten >=1pt,
    every node/.style={font=\small},
    scale=1,
    transform shape
]

%--------------------------------
% vertices
%--------------------------------

\node (w0) at (0,1) {$w^{(0)}$};
\node (w1) at (2,1) {$w^{(1)}$};
\node (w2) at (4,1) {$w^{(2)}$};
\node (w3) at (6,1) {$w^{(3)}$};

\node (v0) at (0,0) {$v^{(0)}$};
\node (v1) at (2,0) {$v^{(1)}$};
\node (v2) at (4,0) {$v^{(2)}$};
\node (v3) at (6,0) {$v^{(3)}$};

\node (z0) at (0,-1) {$z^{(0)}$};
\node (z1) at (2,-1) {$z^{(1)}$};
\node (z2) at (4,-1) {$z^{(2)}$};
\node (z3) at (6,-1) {$z^{(3)}$};

%--------------------------------
% edges v_i -> w_{i+1}
%--------------------------------

\draw[->] (v0) -- (w1);
\draw[->] (v1) -- (w2);
\draw[->] (v2) -- (w3);

%--------------------------------
% edges w_i -> v_{i+1}
%--------------------------------

\draw[->] (w0) -- (v1);
\draw[->] (w1) -- (v2);
\draw[->] (w2) -- (v3);

%--------------------------------
% edges v_i -> z_{i+1}
%--------------------------------

\draw[->] (v0) --  (z1);
\draw[->] (v1) -- (z2);
\draw[->] (v2) -- (z3);

%--------------------------------
% edges z_i -> v_{i+1}
%--------------------------------

\draw[->] (z0) -- (v1);
\draw[->] (z1) -- (v2);
\draw[->] (z2) -- (v3);

%--------------------------------
% dots
%--------------------------------

\node at (-1.2,1) {$\boldsymbol{\cdots}$};
\node at (-1.2,0) {$\boldsymbol{\cdots}$};
\node at (-1.2,-1) {$\boldsymbol{\cdots}$};

\node at (7.2,1) {$\boldsymbol{\cdots}$};
\node at (7.2,0) {$\boldsymbol{\cdots}$};
\node at (7.2,-1) {$\boldsymbol{\cdots}$};

\end{tikzpicture}
\]  
\end{example}

%The following lemma establishes the relationship between the talented monoid \(T_{(E,w)}\) and the covering graph \((\overline{E},\overline{w})\) of a vertex-weighted graph \((E,w)\).

%\begin{lemma} Let \((E,w)\) be a vertex-weighted graph. Then
%\[
%T_{(E,w)} \cong M_{(\overline{E},\overline{w})}.
%\]
%\end{lemma}

\begin{example}
Consider the weighted graph $E$ with $n+k$ loops each of weight $n$.

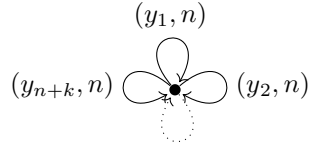
\begin{equation}\label{weightedlpaex}
    \begin{tikzpicture}[scale=2]
\node[circle,fill=black,inner sep=1.5pt] (v) at (0,0) {};

% loops
\draw[->] (v) edge[out=130,in=60,loop] 
node[above] {$(y_1, n)$} ();

\draw[->] (v) edge[out=40,in=-30,loop] 
node[right] {($y{_2}, n)$} ();

\draw[dotted, ->] (v) 
edge[out=-50,in=-120,loop]();

\draw[->] (v) edge[out=140,in=210,loop] 
node[left] {$(y_{n+k},n)$} ();
\end{tikzpicture}
\end{equation}
It is easy to see that $M_{(E,w)}$ is the monogenic monoid 
\[M_{(n,n+k)}=\big\langle v \mid nv=(n+k)v \big \rangle.\]
The weighted covering graph $E\times \mathbb Z$ takes the form
\[
\begin{tikzpicture}[
    >=Stealth,
    shorten >=1pt,
    every node/.style={font=\small},
    scale=1,
    transform shape
]

\tikzset{
  multiedge/.style={->, shorten >=1pt, shorten <=1pt}
}

%--------------------------------
% vertices
%--------------------------------

\node (v0) at (0,1) {$v^{(0)}$};
\node (v1) at (2,1) {$v^{(1)}$};
\node (v2) at (4,1) {$v^{(2)}$};
\node (v3) at (6,1) {$v^{(3)}$};

%--------------------------------
% edges v_i -> w_{i+1}
%--------------------------------

%\draw[->] (v0) -- (v1);
\draw[->, transform canvas={yshift=5pt}] (v0) -- node[midway,above] {$y_1^{(0)}$} (v1);
\node () at (1,1.05) {$\scalebox{.6}{\boldsymbol{\vdots}}$};
\draw[->, transform canvas={yshift=-5pt}] (v0) -- node[midway,below] {$y_{n+k}^{(0)}$} (v1);

%\draw[->] (v1) -- (v2);
\draw[->, transform canvas={yshift=5pt}] (v1) -- node[midway,above] {$y_1^{(1)}$} (v2);
\node () at (3,1.05) {$\scalebox{.6}{\boldsymbol{\vdots}}$};
\draw[->, transform canvas={yshift=-5pt}] (v1) -- node[midway,below] {$y_{n+k}^{(1)}$} (v2);

%\draw[->] (v2) -- (v3);
\draw[->, transform canvas={yshift=5pt}] (v2) -- node[midway,above] {$y_1^{(2)}$} (v3);
\node () at (5,1.05) {$\scalebox{.6}{\boldsymbol{\vdots}}$};
\draw[->, transform canvas={yshift=-5pt}] (v2) -- node[midway,below] {$y_{n+k}^{(2)}$} (v3);

\node at (-1.2,1) {$\boldsymbol{\cdots}$};

\node at (7.2,1) {$\boldsymbol{\cdots}$};

\end{tikzpicture}
\]
%$$\mathcal V^{\gr}(L_\K(n,n+k))\cong \mathcal V^{\gr}(\K)/\langle [\K^n]=[\K^{n+k}(-1)]\rangle.$$ 
where the edges have weight $n$. We then have 
\[T_{(E,w)}=\big \langle v(i), i\in \mathbb Z\mid  (n+k)v(i)=nv(i+1)\big \rangle.\]
Define the $\mathbb Z$-monoid homomorphism 
\begin{align*}
\phi: T_E &\longrightarrow
\left\langle
\left(\frac{n+k}{n}\right)^i \mid  i\in\mathbb Z
\right\rangle\subseteq \mathbb Q^{+}\\
v(i) &\longmapsto \left(\frac{n+k}{n}\right)^i,
\end{align*}
with the action of $\mathbb Z$ defined as follows: 
\[{}^1x=\left(\frac{n+k}{n}\right) x.\]
When $(n,k)=1$, note that the smallest $p$ and $q$ such that 
$p(\frac{n+k}{n})^i=q (\frac{n+k}{n})^{i+1}$ are $n+k$ and $n$, respectively. It then  follows that $\phi$ is an isomorphism. 

\end{example}

Let $M$ be a $\Gamma$-monoid. A {\it $\Gamma$-ordered ideal} of $M$ is a submonoid $I$ of $M$ which is closed under the action of $\Gamma$ and  and it is hereditary in the sense that $x \le y$ and $y \in I$ imply $x\in I$. 

We end this section by 
giving connections between the graph monoid $M_{(E,w)}$ of a vertex-weighted graph $(E, w)$ and the talented monoid $T_{(E,w)}$. 

\begin{prop}\label{goodviasm}
Let $(E,w)$ and $(F,w)$ be vertex-weighted graphs. Then the following statements hold:
    \begin{enumerate}[\upshape(1)]
\item There exists a $\mathbb{Z}$-isomorphism   $\phi: T_{(E, w)}\longrightarrow M_{(\overline E, \overline w)}$ such that $\phi(v(n)) = v^{(n)}$ for all $v\in E^0$ and $n\in \mathbb{Z}$.

\item The talented monoid $T_{(E,w)}$ is cancellative.

\item  A $\mathbb Z$-homomorphism $\phi:T_{(E,w)}\longrightarrow T_{(F,w)}$ induces a homomorphism  $\overline{\phi}: M_{(E,w)}\longrightarrow  M_{(F,w)}$. 

\item Furthermore, if $\phi$ in \textnormal{(3)} is an isomorphism, so it $\overline{\phi}$.

\item There is a lattice isomorphism between the lattices of all order ideals of $M_E$ and $M_{(E,w)}$, and the lattice of all $\mathbb Z$-order ideals of $T_{(E,w)}$.
 \end{enumerate} 
 \end{prop}
\begin{proof}

(1) It is a straightforward check and is left to the reader.

(2) It is obvious that the covering graph $(\overline E, \overline w)$ is acyclic. By item (1), $T_{(E, w)}\cong M_{(\overline E, \overline w)}$ as $\mathbb{Z}$-monoids. Then, by Theorem \ref{prop:cancel-gmonoid}, we immediately obtain that $T_{(E, w)}$ is cancellative.

(3) First observe that there is a monoid isomorphism 
\[
M_{(E,w)} \longrightarrow T_{(E,w)}/ \langle v(i)=v(i+1) \mid v\in E^0, i\in \mathbb Z\rangle.\]
Let $\phi:T_{(E,w)}\rightarrow T_{(F,w)}$ be the $\mathbb Z$-monoid homomorphism. Since $\phi$ preserves the action of $\mathbb Z$, it induces the homomorphism 
\[
T_{(E,w)}/ \langle v(i)=v(i+1) \mid v\in E^0, i\in \mathbb Z\rangle
\longrightarrow T_{(F,w)}/ \langle v(i)=v(i+1) \mid v\in F^0, i\in \mathbb Z\rangle.\]
Thus $\overline \phi : M_{(E,w)}\rightarrow M_{(F,w)}$ follows.  The rest is immediate.

(4) This follows from the form of the maps $\phi$.

(5) A proof analogous to that of Theorem~\ref{sat-hered-ord-id-lat} shows that there is a lattice isomorphism between the lattice of all order ideals of each of these monoids and the lattice of all hereditary and saturated subsets of $E^0$.
\end{proof}

\section{Weighted Leavitt path algebras and the graded Grothendieck group \texorpdfstring{$K_0^{\mathrm{gr}}$}{K0gr}}\label{distin}
The aim of this section is to compute the graded Grothendieck group of the Leavitt path algebra of a vertex-weighted graph in terms of its talented monoid (Corollary \ref{cor:smashprod}), and show that the graded Grothendieck group $K^{\gr}_0$ distinguishes the class of vertex-weighted Leavitt path algebras from the unweighted class (Theorem \ref{mainresult-sec3}), as well as prove that every order preserving $\mathbb{Z}$-module isomorphism between $K_0^{\gr}(L_{\K}(E, w))$ and $K_0^{\gr}(L_{\K}(F, w))$ provides an isomorphism between the semilattices of all vertex-generated ideals of $L_{\K}(E, w)$ and $L_{\K}(F, w)$ (Theorem \ref{mainesult-sec3.2}).
\medskip

We begin this section by briefly recalling the notion of weighted Leavitt path algebras, which are algebras associated with weighted graphs. We refer the reader to \cite{H, Pre}  for a detailed analysis of these algebras. 
\begin{deff}\label{weighteddef}
Let $(E,w)$ be a weighted graph and $\K$ a field.  The free $\K$-algebra generated by $\{v,e_i,e_i^*\mid v\in E^0, e\in E^1, 1\leq i\leq w(e)\}$ subject to relations
\begin{enumerate}[(i)]
\item $uv=\delta_{uv}u,  \text{ where } u,v\in E^0$,
\medskip
\item $s(e)e_i=e_i=e_ir(e),~r(e)e_i^*=e_i^*=e_i^*s(e),  \text{ where } e\in E^1, 1\leq i\leq w(e)$,
\medskip
\item 
$\sum_{e\in s^{-1}(v)}e_ie_j^*= \delta_{ij}v, \text{ where } v\in E_{\reg}^0 \text{ and } 1\leq i, j\leq w(v)$, 
\medskip 
\item $\sum_{1\leq i\leq w(v)}e_i^*f_i= \delta_{ef}r(e), \text{ where } v\in E_{\reg}^0 \text{ and } e,f\in s^{-1}(v)$,
\end{enumerate}
is called the {\it weighted Leavitt path algebra} of $(E,w)$, and denoted by $L_\K(E,w)$, where $\delta$ is the Kronecker delta. In relations (iii) and (iv) we set $e_i$ and $e_i^*$ to be zero whenever $i > w(e)$.    
\end{deff}

Note that if the weight of each edge is $1$, then $L_\K(E,w)$ reduces to the usual Leavitt path algebra $L_\K(E)$  of the graph $E$.   Also, if $E$ is the empty graph then $L_\K(E,w)$ is defined to be the zero ring.  

The motivation for defining weighted Leavitt path algebras comes from the fact that ordinary Leavitt path algebras do not encompass Leavitt algebras of the form $L_\K(n,n+k)$. However, starting from the weighted graphs in (\ref{weightedlpaex}), the associated weighted Leavitt path algebras are precisely $L_\K(n,n+k)$.

For a vertex-weighted graph $(E,w)$, one can consider a $\mathbb Z$-grading on
$L_\K(E,w)$ by assigning
\[
\deg(v)=0,\qquad
\deg(e_i)=1,\qquad
\deg(e_i^*)=-1,
\]
for $v\in E^0$, $e\in E^1$, and $1\le i\le w(e)$. If all edges have weight
$1$, then this grading coincides with the standard grading on an unweighted
Leavitt path algebra. For an approach to endowing weighted Leavitt path
algebras with a grading by an arbitrary group, we refer the reader to
\cite[Lemma~10.2.3]{Pre}.

Let $R$ be a not necessarily unital ring. Recall that a left $R$-module $M$ is called {\it unital} if $RM = M$. We denote by $R$-$\mathrm{Mod}$  the category of unital left $R$-modules. Furthermore, we denote by $R$-$\mathrm{Mod}_{\mathrm{proj}}$  the full subcategory of $R$-$\mathrm{Mod}$ whose objects are the projective objects of $R$-$\mathrm{Mod}$ that are finitely generated as a left $R$-module. If $R$ has local units, we define
\[ \V(R) = \{ [P] ~| ~ P \in  R\text{-}\mathrm{Mod}_{\mathrm{proj}}   \} \]
where $[P]$ denotes the isomorphism class of $P$ as a left $R$-module. $\V(R)$ becomes an abelian monoid by defining $[P]+[Q] = [P \oplus Q]$. It is well-known that $\V$ is a functor that commutes with direct limits. There is a different, but equivalent definition of $\V$ using idempotent matrices over $R$, cf. \cite[Subsection 4A]{ara-hazrat-li-sims}. 

The following result was introduced by the second author in \cite{H} and later generalised to arbitrary weighted graphs by R. Preusser in \cite{P}.

\begin{thm}[{\cite[Theorem 14]{P}}]\label{VLisoM}
    $\V \circ L \cong M$. Moreover, if $(E,w)$ is
a finite weighted graph and $\K$ is a field, then $L_{\K}(E,w)$ is left and right hereditary.
\end{thm}

Let $R$ now be a $\Gamma$-graded ring. Recall that a $R$-module $M$ is called $\Gamma$-graded if there is a decomposition $M = \oplus_{\gamma \in \Gamma} M_{\gamma}$ such that $R_{\alpha} M_{\gamma} \subseteq M_{\alpha \gamma}$ for any $\alpha \gamma \in \Gamma$. We denote by $R\text{-} \mathrm{Gr}$ the category of $\Gamma$-graded unital left $R$-modules with morphisms the $R$-module homomorphisms that preserve grading. Moreover, we denote by $R\text{-}\mathrm{Gr}_{\mathrm{proj}}$ the full subcategory of $R$-$\mathrm{Gr}$ whose objects are the projective objects of $R\text{-}\mathrm{Gr}$ that are finitely generated as a left $R$-module. If $R$ has graded local units, we define
\[ \V^{\mathrm{gr}} (R) = \{ [P] ~|~ P \in R\text{-}\mathrm{Gr}_{\mathrm{proj}} \} \]
where $[P]$ denotes the isomorphism class of $P$ as a graded left $R$-module. $\V^{\mathrm{gr}} (R)$ becomes a $\Gamma$-monoid by defining $[P]+[Q] = [P \oplus Q]$ and ${^{\gamma}}{[P]} = [P(\gamma)]$. The graded Grothendieck group $K_0^{\mathrm{gr}} (R)$ is the group completion of $\V^{\mathrm{gr}} (R)$ \cite[Subsection 4A]{ara-hazrat-li-sims}.

We next discuss the smash product of weighted Leavitt path algebras. For a $\G$-graded ring $A$ (possibly without unit), the \emph{smash product} ring $A\#\G$
is defined as the set of all formal sums $\sum_{\gamma \in \G} r^{(\gamma)} P_\gamma $,
where $r^{(\gamma)}\in A$ and $P_\gamma$ are symbols. Addition is defined component-wise
and multiplication is defined by linear extension of the rule
$$(rP_{\a})(sP_{\b})=rs_{\a\b^{-1}}P_{\b},$$  where $r,s\in A$ and $\a,\b\in\G$.
It is routine to check that $A\#\G$ is a ring. We emphasise that the symbols $P_{\g}$ do
not belong to $A\#\G$; however if the ring $A$ has unit, then we regard the $P_\g$ as
elements of $A\#\G$ by identifying $1_A P_\g$ with $P_\g$. Each $P_{\g}$ is then an
idempotent element of $A\#\G$.

It is known that (see~\cite{ara-hazrat-li-sims} or \cite[Proposition 66]{P2020}) there is an isomorphism of
categories of modules 
\begin{equation}\label{vast5}
\psi: A\-\mathrm{Gr}\longrightarrow A\#\G\-\mathrm{Mod}.
\end{equation}

For a $\Gamma$-graded ample groupoid $\mathcal G$, in \cite{ara-hazrat-li-sims} it was shown that there is an isomorphism between the Steinberg algebras $A_\K(\mathcal G \times \Gamma)  \cong A_\K(\mathcal G)\# \Gamma$, where $\mathcal G \times \Gamma$ is the crossed-product groupoid. Specialising to the case of graph groupoids $\mathcal G_E$, since $A_\K(\mathcal G_E)\cong L_\K(E)$, we obtain that 
\[L_\K(\overline E) \cong A_\K(\mathcal G_{\overline E}) \cong A_\K(\mathcal G_E \times \mathbb Z) \cong A_\K(\mathcal G_E)\# 
\mathbb Z\cong L_\K(E) \# \mathbb Z. \]  Combined with the category isomorphism of~\ref{vast5}, this  explains why graded $K$-theory is both tractable and powerful in the context of Leavitt path algebras. Although the computations become considerably simpler, the theory still captures substantial information about the algebra itself, as in effect, one is working with the $K$-theory of algebras associated to acyclic graphs, namely the covering graphs.

There is no groupoid model for weighted Leavitt path algebras. However, the following lemma once again highlights the effectiveness of graded $K$-theory in the weighted graph setting. Compare this with  \cite[Corollary 5.3]{ara-hazrat-li-sims} for the case of Leavitt path algebras and~\cite[Proposition 74]{P2020} for Leavitt path algebras of hyper graphs. 

\begin{lemma}\label{lm:smashprod}
For any vertex-weighted graph $(E, w)$ and any field $\K$, the map
$$
\phi \colon L_\K(\overline E, \overline w) \rightarrow L_\K(E,w)\# \mathbb Z
$$ defined on the generators by
\begin{equation}\label{smashweight}
\phi(v^{(n)})=v{P_{-n}}, \qquad
\phi(e_i^{(n)})=e_i P_{-n-1}, \qquad
\phi(e_i{^{(n)}}^*)=e_i^*{P_{-n}},
 \end{equation}
for $v\in E^0$, $e\in E^1$, $1\leq i \leq w(e)$, and $n \in \mathbb Z$,
is an isomorphism of $\K$-algebras.
 \end{lemma}
 \begin{proof}
To show that the map $\phi$ is a well-defined $\K$-algebra homomorphism, it suffices to verify that $\phi$ preserves the defining relations of the weighted Leavitt path algebra of the covering graph. We verify relation~(iii) of Definition~\ref{weighteddef} and leave the remaining verifications to the reader.
Let $v^{(n)}\in \overline{E}^0$. We then have
\begin{align*}
\phi\left(\sum_{e\in s^{-1}(v^{(n)})} e_i^{(n)} (e_j^{(n)})^*\right)
&= \sum_{e\in s^{-1}(v)} \phi(e_i^{(n)})\phi((e_j^{(n)})^*) \\
&= \sum_{e\in s^{-1}(v)} \bigl(e_i P_{-n-1}\bigr)\bigl(e_j^* P_{-n}\bigr) \\
&= \sum_{e\in s^{-1}(v)} e_i e_j^* P_{-n} \\
&= \left(\sum_{e\in s^{-1}(v)} e_i e_j^*\right) P_{-n} \\
&= \delta_{ij}vP_{-n} \\
&= \phi(\delta_{ij}v^{(n)}),
\end{align*} as desired. The image of $\phi$ contains the set $$S := \{vP_n, e_iP_n, e^*_iP_n \mid v \in E^0,\ e \in E^1,\ 1 \le i \le w(e),\ n \in \mathbb{Z}\}.$$ Since $S$ generates $L(E,w)\# \Z$ as a $\K$-algebras, it follows that $\phi$ is surjective. Since weighted Leavitt path algebras have normal forms~\cite{hazrat2017applications}, one can define the well-define inverse map to $\phi$, concluding that $\phi$ is an isomorphism. 
 \end{proof}

Consequently, we obtain the following useful fact.

\begin{cor}\label{cor:smashprod}
Let $(E, w)$ be a vertex-weighted graph and $\K$ a field. Then \[  T_{(E,w)} \cong M_{(\overline{E},\overline{w})} \cong \V(L_\K(\overline{E},\overline{w})) \cong \V^{\gr}(L_\K(E,w)). \]  Consequently, $T_{(E, w)}$ is the positive cone of the graded Grothendieck group of $L_{\K}(E, w)$.
\end{cor}
\begin{proof}
Combining Lemma~\ref{lm:smashprod}, with isomorphism~\ref{vast5} we have an isomorphism of categories 
\begin{equation}\label{vast7}
\psi: L_\K(E,w)\-\mathrm{Gr}\longrightarrow L_\K(\overline E, \overline w)\-\mathrm{Mod}.
\end{equation}
This isomorphism shows that the graded Grothendieck group $K_0^{\gr}(L_\K(E,w))$ coincides with the Grothendieck group 
$K_0(L_\K(\overline E, \overline w))$. Then, by Proposition~\ref{goodviasm}, \cite[Theorem 5.21]{H}, and (\ref{vast7}), we have a sequence of $\Z$-isomophisms \[  T_{(E,w)} \cong M_{(\overline{E},\overline{w})} \cong \V(L_\K(\overline{E},\overline{w})) \cong \V^{\gr}(L_\K(E,w)),\]  thus finishing our proof.
\end{proof}

We recall Preusser's conditions on a weighted Leavitt path algebras to be isomorphic to  an unweighted one~\cite{Pre,Preusser2021}. 

\begin{deff}\label{defLPA}
We say that a weighted graph $(E,w)$  satisfies the {\it Preusser conditions} if the following hold true:
\begin{enumerate}[({P}1)]
\item Any vertex $v\in E^0$ emits at most one weighted edge.
\item Any vertex $v\in T(r(E^1_w))$ emits at most one edge.
\item If two weighted edges $e,f\in E^1_w$ are not in line, then $T(r(e))\cap T(r(f))=\varnothing$.
\item If $e\in E^1_w$ and $c$ is a cycle based at some vertex $v\in T(r(e))$, then $e$ belongs to $c$.
\end{enumerate}
\end{deff}

In \cite{Preusser2021} it was proved that a weighted Leavitt path algebra $L_\K(E,w)$ is isomorphic to an unweighted Leavitt path algebras if and only if Preusser's conditions hold.

We illustrate, by means of an example, how Preusser's conditions allow one to construct an unweighted graph whose associated Leavitt path algebra is (graded) isomorphic to that of a weighted graph. We emphasis that to preserve the grading, however, the canonical grading on edges must be modified, as the example below demonstrates. Consider the weighted graph $(E,w)$:
\[
\begin{tikzpicture}[
    >=Stealth,
    shorten >=1pt,
    every node/.style={font=\small},
    scale=1,
    transform shape
]
\tikzset{
  multiedge/.style={->, shorten >=1pt, shorten <=1pt}
}

%--------------------------------
% vertices
%--------------------------------

\node (v0) at (0,1) {$u$};
\node (v1) at (2,1) {$v$};
\node (v2) at (4,1) {$w$};

\draw[->] (v0) -- node[midway,above] {$(e,2)$} (v1);
\draw[->] (v1) -- node[midway,above] {$(f,2)$} (v2);

\end{tikzpicture}
\]

Then following~\cite{Preusser2021}, we first obtain the intermediate weighted graph $(\widetilde{E},w)$
\[
\begin{tikzpicture}[
    >=Stealth,
    shorten >=1pt,
    every node/.style={font=\small},
    scale=1,
    transform shape
]

\tikzset{
  multiedge/.style={->, shorten >=1pt, shorten <=1pt}
}

%--------------------------------
% vertices
%--------------------------------

\node (v0) at (0,1) {$u$};
\node (v1) at (2,1) {$v$};
\node (v2) at (4,1) {$w$};

\draw[->] (v0) -- node[midway,above] {$(e,2)$} (v1);
\draw[->,  transform canvas={yshift=2pt}] (v2) -- node[midway,above] {$f_1$} (v1);
\draw[->,  transform canvas={yshift=-2pt}] (v2) -- node[midway,below] {$f_2$} (v1);
\end{tikzpicture}
\]

The assignment that sends vertices of $E$  to the corresponding ones in $\widetilde{E}$, $e_i\mapsto e_i$, $f_i\mapsto f_i^*$, and ghost edges accordingly, induces an isomorphism $\phi: L_\K(E,w)\cong L_\K(\widetilde{E},w)$. Thus if we keep the canonical grading for $E$, and in $\widetilde{E}$ we assign $\deg(e_i)=1$ and $\deg(f_i)=-1$, then the map $\phi$  is indeed a graded isomorphism.  Next we derive a unweighted graph $F$ from $\widetilde{E}$ as follows: 
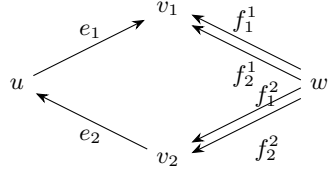
\[
\begin{tikzpicture}[
    >=Stealth,
    shorten >=1pt,
    every node/.style={font=\small},
    scale=1,
    transform shape
]

\tikzset{
  multiedge/.style={->, shorten >=1pt, shorten <=1pt}
}

%--------------------------------
% vertices
%--------------------------------

\node (v0) at (0,1) {$u$};
\node (v1) at (2,2) {$v_1$};
\node (v11) at (2,0) {$v_2$};
\node (v2) at (4,1) {$w$};

\draw[->] (v0) -- node[midway,above] {$e_1$} (v1);
\draw[->] (v11) -- node[midway,below] {$e_2$} (v0);
\draw[->,  transform canvas={yshift=2pt}] (v2) -- node[midway, above] {$f^1_1$} (v1);
\draw[->,  transform canvas={yshift=-2pt}] (v2) -- node[midway,below] {$f^1_2$} (v1);

\draw[->,  transform canvas={yshift=2pt}] (v2) -- node[midway,above right] {$f^2_1$} (v11);
\draw[->,  transform canvas={yshift=-2pt}] (v2) -- node[midway, below right] {$f^2_2$} (v11);

\end{tikzpicture}
\]
The assignment $u\mapsto u$, $w\mapsto w$ and $v\mapsto v_1+v_2$ on vertices and  $e_1\mapsto e_1$, $e_2\mapsto e_2^*$, $f_i\mapsto f_i^1+f_i^2$ on edges induce an isomorphism $L_\K(\widetilde{E},w)\cong L_\K(F)$. Thus if in the graph $F$ we assign $\deg(e_1)=1$, $\deg(e_2)=-1$ and $\deg(f^1_i)=\deg(f^2_i)=-1$, then above map is indeed a graded isomorphism. Putting these together we have \[L_\K(E,w)\cong_{\gr} L_\K(\widetilde{E},w) \cong_{\gr} L_\K(F).\]

We are in a position to prove one of the main results of the note. 

\begin{prop}\label{cortal}
   Let $(E,w)$ be a vertex-weighted graph. If $(E,w)$ does not satisfy Preusser's condition, then $T_{(E,w)}$ fails to have the refinement property. 
\end{prop}
\begin{proof}
We check that if any of four Preusser's conditions~(\ref{defLPA}) fails, then $T_{(E,w)}$ is not refinement. 

1. Suppose (P1) fails. Since the graph is a vertex-weighted graph, then there is  a vertex $v$ which emits at least two edges of weight $n>1$. Thus in $T_{(E,w)}$ we have $nv(i)=\sum_{\{e\in s^{-1}(v)\}} r(e)(i+1)$, where the sum consists of more than one item. Thus $nv=w_1(1)+\cdots + w_k(1)$, where $w_k\in E^0$ and $k>1$ and $n>1$. Suppose $T_{(E,w)}$ is a refinement monoid. Then there exist elements $a_{ij} \in T_{(E,w)}$, $1\leq i \leq n$, $1 \leq j \leq  k$,  such that 
    \begin{align*}
     v&=a_{11}+\cdots +a_{1k}\\ 
      v&=a_{21}+\cdots +a_{2k}\\ 
     &\vdots \\
     v&=a_{n1}+\cdots +a_{nk}\\
     w_i(1)&=a_{1i}+\cdots +a_{ni} \text{ where } 1\leq i \leq k \\
     \end{align*}
    Since $v=a_{i1}+\cdots +a_{ik}$, $1\leq i \leq n$, the Confluence Lemma~\ref{aralem6} implies that the both sides of the equality flow to the same element in the free monoid. But  the only transformation allowed on $v$ is when we have $nv$, $n>1$, thus the confluence property implies that $a_{i1}+\cdots +a_{ik}\rightarrow v$. Since we are working on the covering graph, all elements $a_{ij}$ first need to  flow to level $-1$ before the last transformation to $v$. Now transformation of vertices from level $-1$ to  a single vertex $v$ in the free commutative monoid, requires there are weighted edge $f_i$, of weight $l_i$, $1\leq i \leq n$ such that $s(f_i)=z_i$, $r(f_i)=v$ and all $a_{ij}$, $1\leq j \leq k$ transformed to $z_i(-1)$, so that from level -1 to $0$ they will be flown to the single vertex $v$, by a transformation of the form  $l_iz_i(-1) \rightarrow v$. In particular, if $a_{i1}\not = 0$, then it will  flow to some copies of $z_i(-1)$, i.e., $a_{i1}=l'_iz_i(-1)$ in $T_{(E, w)}$, $l'_1\in \mathbb N$ and $1\leq i \leq n$.  Note that the vertices $z_i$  are not necessarily distinct. Since $a_{i1}=l'_iz_i(-1)$  and $l_iz_i(-1)=v$ in $T_{(E, w)}$, it follows that $\sum_{1\leq i \leq n}a_{i1} = m v$ in $T_{(E, w)}$ for some $m\ge 1$.  Since $\sum_{1\leq i \leq n}a_{i1}= w_1(1)$, we obtain that $w_1(1) =  m v$ in $T_{(E, w)}$. By Confluence Lemma~\ref{aralem6},  both $w_1(1)$ and $mv$  flow to the same element in the free monoid. This implies that $m \ge n$. We then have $$w_1(1) = m v = nv + (m-n)v = w_1(1)+\cdots + w_k(1) + (m-n)v= w_1(1) + x$$  in $T_{(E, w)}$, where $x = w_2(1)+\cdots + w_k(1) + (m-n)v$. Since $k > 1$, we must have $x \neq 0$. On the other hand, by Proposition~\ref{goodviasm}(2), $T_{(E,w)}$ is cancellative. Therefore, the equality $w_1(1)=w_1(1)+x$ would imply  $x =0$, a contradiction, and so $T_{(E,w)}$ is not refinement. 
    
   2. Suppose (P2) fails. Then there is a vertex $z\in T(r(E^1_w))$ that emits more than one edge. Hence there is a weighted edge $f$, $w(f)=n>1$, such that $v=s(f)$ and the $r(f)\geq z$. In $T_{(E,w)}$ we then have $nv=w_1(i_1)+\cdots + w_k(i_k)$, where $w_k\in E^0$ and $i_k, k>1$ and $n>1$. A similar argument as in part (1) shows that $T_{(E,w})$ is not refinement.

 3. Suppose (P3) fails. Then there are two weighted edges $e,f\in E^1_w$ which are not in line, but $T(r(e))\cap T(r(f))\not =\varnothing$. If $s(e)=s(f)$ then P(1) fails, and thus $T_{(E,w)}$ is not refinement by part 1. So suppose $v:=s(e)$ and $w:=s(f)$ and $v\not = w$. 
 If P(2) also fails then $T_{(E,w)}$ is already not refinement. Thus there are paths without bifurcation from $r(e)$ and $r(f)$ to a vertex $z$. It follows that there are $n,m>1$, such that $nv(i)=kw(j)$, where $i,j \in \mathbb N$. 
 Without loss of generality, we can assume that $nv=kw(j)$, for some $j\in \mathbb N$. The rest of the argument is rather similar to part 1. If $T_{(E,w)}$ is refinement, we then have 
\begin{align*}
     v&=a_{11}+\cdots +a_{1k}\\ 
      v&=a_{21}+\cdots +a_{2k}\\ 
     &\vdots \\
     v&=a_{n1}+\cdots +a_{nk}\\
     w(j)&=a_{1i}+\cdots +a_{ni} \text{ where } 1\leq i \leq k \\
     \end{align*}
As there is no transformation on $v$ allowed, Confluence Lemma~\ref{aralem6}  gives that $a_{i1}+\cdots +a_{ik} \rightarrow v$. Thus there should be an edge $f$ with $s(f)=t$ and $r(f)=v$ and all $a_{ij}$ flows to $t(-1)$ which after the next transformation they arrive to $v$. However it is then impossible that the same $a_{ij}$ on the level -1, flow to $w$ on level $0$ as $v\not = w$. A contradiction, and thus $T_{(E,w)}$ is not refinement. 
    
Suppose (P4) fails. Then there is an edge $e\in E^1_w$ and a cycle $c$ based at some vertex $w\in T(r(e))$, such that $e$ does not belongs to $c$. If any of the conditions P(1) to P(3) fail, then $T_{(E,w)}$ is not refinement and we are done. So assume conditions P(1) to P(3) hold. Thus there is a path without bifurcation from $v:=s(e)$ to $w$ the base of the cycle $c$. Hence, in $T_{(E,w)}$, we have $nv=w(k)$, for some $k\in \mathbb N$. The P(1) to P(4) conditions guarantee the cycle $c$ has no exits, and there is no path returning to $v$. Therefore there is a natural number $m$ such that $mw=w(k')$ for some $k'>1$. Combining these two equations we have 
  \[nm v= mw(k)=w(k+k')=nv(k').\]
If $T_{(E,w)}$ is refinement, then there are $a_{ij}$ that flows from level -1 to $v$ at level 0. But then the same $a_{ij}$ should flow further to level $k'>1$ to reach to $v(k')$. This means that $v$ is on a cycle. But $e$ was not part of the cycle $c$, meaning, $T(r(e))$ emits more than one edge, a contradiction. 
\end{proof}

As a corollary of Proposition \ref{cortal} and \cite[Theorem 5.2.1 ]{Pre}), we obtain the following interesting result.

\begin{thm}\label{mainresult-sec3}
The graded Grothendieck group $K^{\gr}_0$ distinguishes the class of vertex-weighted Leavitt path algebras from the unweighted class.
\end{thm}
\begin{proof}
Let $L_\K(E,w)$ be a weighted Leavitt path algebra. Suppose there is an unweighted graph $F$ such that there is an order preserving  $\mathbb Z[x,x^{-1}]$-module isomorphism $K_0^{\gr}(L_\K(E,w))\cong K_0^{\gr}(L_\K(F))$. Since the positive cones of $K_0^{\gr}$-groups coincide with the talented monoids, we have   $T_{(E,w)} \cong T_F$. But the monoid $T_F$ is refinement. Therefore, by Proposition~\ref{cortal}, the weighted graph $(E,w)$ satisfies  Preusser's conditions. Thus $L_\K(E,w)$ is isomorphic to a unweighted Leavitt path algebras. 
In fact one can establish a $\mathbb Z$-graded isomorphism, by assigning appropriate grading to edges (see the isomorphism between the Leavitt path algebras in Theorem 5.2.1 of~\cite{Pre}).  

On the other hand, suppose that 
$K_0^{\gr}(L_\K(E,w))$  cannot be realised as  $K_0^{\gr}(L_\K(F))$ for any unweighted graph $F$. Then $L_\K(E,w)$ cannot be graded isomorphic to a unweighted Leavitt path algebra. So it cannot be isomorphic to a Leavitt path algebra. Otherwise Preusser's condition implies that $L(E,w)$ is isomorphic with a Leavitt path algebra of $L_\K(F')$ for certain graph $F'$ and with appropriate grading they are in fact $\mathbb Z$-graded isomorphic, which is a contradiction. 
\end{proof}

\begin{lemma}\label{hfgyhf}
Let $(E,w)$ be a vertex-weighted graph and $\K$ a field.  Then, there is a semilattice isomorphism between the  semilattices of hereditary and saturated subsets of $E^0$ and the (graded) ideals of $L_\K(E,w)$ generated by vertices. 
\end{lemma}
\begin{proof}
We denote by $\mathcal{L}_{\text{ver}}(L_{\K}(E, \omega))$ the set of all ideals of $L_{\K}(E, \omega)$ generated by vertices. It is obvious that $(\mathcal{L}_{\text{ver}}(L_{\K}(E, \omega)), \subseteq)$ is a join-semilattice. 	We next claim that \[\mathcal{H}_{(E, \omega)}\cong \mathcal{L}_{\text{ver}}(L_{\K}(E, \omega))\] as lattices. Indeed, let $I$ be a nonzero ideal of $L_{\K}(E, \omega)$ generated by vertices. It is obvious that $I \cap E^0\neq \emptyset$.
We show that $I\cap E^0$ is a hereditary and saturated subset of $E^0$. Let $e\in E^1$ with $v:=s(e)\in I$. We then have $e_i = ve_i\in I$ for all  $1\le i\le \omega(e)$, and so $r(e)=\sum_{1\leq i\leq \omega(v)}e_i^*e_i\in I$. This implies that $I\cap E^0$ is hereditary. Let $v\in E_{\reg}^0$ with the property that $r(s^{-1}(v))\subseteq I\cap E^0$. We then have $r(e_1)\in I$ for all $e\in s^{-1}(v)$, and so $e_1 = e_1 r(e_1)\in I$ for all $e\in s^{-1}(v)$. Therefore, we obtain that $v = \sum_{e\in s^{-1}(v)}e_1e_1^*\in I$, and so $I\cap E^0$ is saturated, as desired.
	
On the other hand if $H$ is hereditary and saturated and $I(H)$ is an ideal of $L(E,\omega)$ generated by $H$, then $I(H) \cap E^0= H$ by 
\cite[Theorem 2.10]{haznam}. From these observations, we immediately obtain that the corresponding $\alpha: \mathcal{L}_{\text{ver}}(L_{\K}(E, \omega))\longrightarrow \mathcal{H}_{(E, \omega)}$, defined by $I\longmapsto I\cap E^0$, is a join-semilattice isomorphism, thus finishing the proof.
\end{proof}

We end this section by the following interesting result.

\begin{thm}\label{mainesult-sec3.2}
Let $(E,w)$ and $(F,w)$ be vertex-weighted graphs and $\K$ a field. An order preserving $\Z[x,x^{-1}]$-module  isomorphism $$K^{\gr}_0(L_\K(E,w))\rightarrow  K^{\gr}_0(L_\K(F,w)),$$ gives a join-semilattice isomorphism between ideals of 
$L_\K(E,w)$ and $L_\K(F,w)$  
generated by vertices.
\end{thm}
\begin{proof}
By Proposition~\ref{goodviasm}(2), and Corollary \ref{cor:smashprod}, the isomorphism of $K^{\gr}_0$-groups induces a $\mathbb Z$-isomorphism $T_{(E,w)}\cong T_{(F,w)}$. Thus the lattice of $\mathbb{Z}$-order ideals of these monoids are isomorphic. By Proposition~\ref{goodviasm}, it follows that the lattices of hereditary and saturated subsets of $E$ and $F$ are isomorphic. Now Lemma~\ref{hfgyhf} implies that the semilattice of ideals of $L_\K(E,w)$ and $L_\K(F,w)$  
generated by vertices are isomorphic, thus finishing the proof. 
\end{proof}

\section{\texorpdfstring{$K_0^{\mathrm{gr}}$}{K0gr} is a full functor}\label{lifting}

The main aim of this section is to prove the following lifting theorem, which gives that Conjecture~\ref{conjalg}(1) holds for vertex-weighted graphs (Theorem \ref{weightedlift}), and show that the graded Grothendieck group $K_0^{\gr}$ classifies the Leavitt algebras $L_\K(n,n+k)$ (Theorem \ref{thm:classify-gLAs}).
\medskip

\begin{thm}\label{weightedlift}
Let $(E,w)$ and
$(F,w)$ be finite vertex-weighted graphs and $\K$ a field.  For any order preserving $\Z[x,x^{-1}]$-module homomorphism $\phi: K^{\gr}_0(L_\K(E,w)) \rightarrow K^{\gr}_0(L_\K(F,w))$ with 
$\phi([L_\K(E,w)])=L_\K(F,w)$, there exists a unital $\mathbb Z$-graded $\K$-homomorphism $\psi: L_\K(E,w) \rightarrow L_\K(F,w)$ such that $K^{\gr}_0(\psi) = \phi$.
\end{thm}

In order to prove this theorem, we need to recall the graded version of Bergman machinery developed in \cite{HazLiP}. 
In that work, Bergman's universal constructions were extended to the setting of graded rings. It was shown that any conical $\Gamma$-monoid can be realised as a graded non-stable $K$-theory of a $\Gamma$-graded ring, which satisfies a ``weak" universal property. We collect here the results from \cite{HazLiP} that will be needed in the subsequent sections.

\begin{thm} \label{thm-1} Let $R$ be a $\Gamma$-graded $\K$-algebra. Suppose that $M$ is a $\Gamma$-graded $R$-module and $P$ a graded finitely generated projective $R$-module. Then there exists a $\Gamma$-graded $R\text{-ring}_{\K}$, $S$, with a universal graded module homomorphism $f: S\otimes_R M \rightarrow{} S\otimes_R P$; that is, given any $\G$-graded $R\text{-ring}_{\K}$, $T$, and any graded $T$-module homomorphism $g:T\otimes_R M\rightarrow{} T\otimes_R P$, there exists a unique graded homomorphism $S\rightarrow{} T$ of $R\text{-rings}_{\K}$ such that $g=T\otimes_S f$. 
\end{thm}

\begin{thm}
\label{thm-2}
Let $R$ be a $\G$-graded $\K$-algebra, $M$ a graded $R$-module, $P$ a graded projective $R$-module, and $f:M\xrightarrow{} P$ a graded module homomorphism.  Then there exists a $\G$-graded $R\text{-ring}_{\K}$ $S$ such that $S\otimes_R f=0$, and $S$ is universal for this property: Given any $\G$-graded $R\text{-ring}_{\K}$ $T$ with $T\otimes_R f=0$, there exists a unique graded homomorphism of $R\text{-rings}_{\K}$, $S\xrightarrow{} T$.
\end{thm}

For graded finitely generated  projective $R$-modules $P$ and $Q$, we can adjoin a universal graded isomorphism between $\overline{P}$ and $\overline{Q}$ by first freely adjoining a graded map $h: \overline{P}\xrightarrow[]{}\overline{Q}$, and $\overline{h}: \overline{Q}\xrightarrow[]{}\overline{P}$ using Theorem~\ref{thm-1} and then using Theorem~\ref{thm-2} to force that $1-h\overline{h}=0$ and $1-\overline{h}h=0$. We denote the resulting $\Gamma$-graded $R\text{-ring}_{\K}$ by \[S:= R\langle h, h^{-1} \mid \overline{P}\cong_{\gr} \overline{Q}\rangle.\]

 When the field $\K$ is concentrated in degree zero, we can compute the non-stable graded $K$-theory of the graded ring $S$ (\cite[\S 7]{HazLiP}). 
 
\begin{thm}\label{thmbergV}
Let $R$ be a $\G$-graded $\K$-algebra, where $\K$ is concentrated in degree zero. Let $P$ and $Q$ be nonzero graded finitely generated  projective $R$-modules and $S:= R\langle h, h^{-1} \mid \overline{P}\cong_{\gr} \overline{Q}\rangle.$
Then is a $\Gamma$-monoid isomorphism \[\V^{\gr}(S) \cong \V^{\gr}(R) \big /\big \langle [P]=[Q]\big \rangle\] given by tensoring. 
\end{thm}

We are now in a position to prove Theorem~\ref{weightedlift}. In fact, by using the graded Bergman construction, we obtain a more general result.
 
\begin{thm}\label{arnonevas}
Let $E$ be a finite vertex-weighted graph, $\K$ a field, and $A$ a unital $\mathbb Z$-graded $\K$-algebra. Let $\phi: T_{(E,w)} \rightarrow \mathcal V^{\gr}(A)$ be a $\mathbb Z$-monoid homomorphism with  
$\phi(1_E)=[A]$. Then there exists a unital $\mathbb Z$-graded $\K$-algebra homomorphism $\psi: L_\K(E,w) \rightarrow A$ such that $\mathcal V^{\gr}(\psi) = \phi$.  

\end{thm}
\begin{proof}
Consider the semisimple $\K$-algebra $R=\prod_{E^0}\K$, the product of $|E^0|$-copies  of the field $\K$ as a $\mathbb Z$-graded algebra concentrated in degree zero.  
Then 
\begin{equation}\label{timemm}
\mathcal V^{\gr}(R) \cong \prod_{E^0} \mathcal V^{\gr}(\K) \cong \big \langle v(i)\mid v\in E^0, \, i \in \mathbb Z \big \rangle,
\end{equation}
with the $\mathbb Z$-action ${}^n v(i)=v(i+n)$, for $i,n \in \mathbb Z$. 
Denote $p_v(i)$, where $v\in E^0$ and $i\in \mathbb Z$, the graded finitely generated projective $R$-module with $\K(i)$ appears in $v$-th component of $R$ and zero elsewhere. Throughout, we write $p_v$ for $p_v(0)$. Note that $\bigoplus_{v\in E^0}p_v \cong R$ as graded $R$-module and the isomorphism classes $[p_v(i)]$ correspond to $v(i)$ in (\ref{timemm}).

Let $\{v_1,\dots,v_m\}$ be the set of all the vertices in the weighted graph $E$ which emit edges.  Consider the following finitely generated graded projective $R$-modules: 
\[P:=\bigoplus_{w(v_1)}p_{v_1}\,\,   \text{ and } \, \,  Q:=\bigoplus_{\{e \in E^1 \mid s(e)=v_1 \}}p_{r(e)}.\]  
We construct the graded algebra
\begin{equation}\label{sequi}
    S_1:=R\big\langle h_{v_1},h_{v_1}^{-1}:\overline{P}\cong \overline{Q(1)}\big\rangle.
    \end{equation}

First using  Theorem~\ref{thm-1}, one constructs a graded $\K$-algebra $S_1'$ and universal homomorphisms $i:S_1' \otimes _{R} P\rightarrow S_1' \otimes _{R} Q(1)$ and $\overline i: S_1' \otimes _{R} Q(1)\rightarrow S_1' \otimes _{R} P$. Next we apply Theorem~\ref{thm-2}  to obtain the ring $S_1$, where extensions of $i$ and $\overline i$ (call them again by $i$ and $\overline i$) over $S_1$ give $1-i\overline i=0$ and $1-\overline i i=0$. Thus we have  $S_1=R \langle i,i^{-1}: \overline P \cong \overline {Q(1)}\rangle$ with a universal isomorphism \[i: \  \overline P:=S_1\otimes_{R} P \ \  \longrightarrow  \ \ \overline  {Q(1)}:=S_1\otimes_{R}Q(1).\] 

The proofs of Theorem~\ref{thm-1} and \ref{thm-2}  show that $S_1$ is $L_\K(X_1,w)$, where $X_1$ is a vertex-weighted graph with the same vertices as $E$ and where $v_1$ emits the same edges as in $E$ and other vertices do not emit any edges. Namely, if $\{ e_1,\dots,e_s\}$ is all the edges which are emitted from $v_1$ with $n=w(v_1)$ then the right multiplication by  the matrix  $Y=(e_{ij})_{1\leq j \leq s, 1\leq i\leq n}$, where $e_{ij}={(e_j)}_i$, gives the map 
\[i: \ \overline P=\bigoplus_{w(v_1)}S_1 v_1 \ \  \longrightarrow \ \  \overline{Q(1)} =\bigoplus_{\{e \in E^1 \mid s(e)=v_1 \}}S_1r(e)(1),\] while $X=(Y^*)^t$, where ${}^t$ is the transpose operation,  gives $i^{-1}$. Now Theorem~\ref{thmbergV} guarantees that $\V^{\gr}(S_1)$ is obtained from $\V^{\gr}(R)$ by adding the relation $[P]=[Q(1)]$.  Translating this to our setting, we get that $\V^{\gr}(S_1)$ is the $\mathbb Z$-monoid generated by the set $\{v(i) \mid v\in E^0, i \in \mathbb Z\}$ subject to the relation 
\[w(v_1) v_1=\sum_{\{e \in E^1 \mid s(e )=v_1 \}} r(e)(1).\] 

We repeat this process to cover the whole graph. To be precise, let $S_k=L_\K(X_k,w)$, $k\geq 1$, where $X_k$ is the graph with the same vertices as $E$, but only the first $k$ vertices $\{v_1,\dots,v_k\}$ emit structured edges. By induction, $\V^{\gr}(S_k)$ is a $\mathbb Z$-commutative monoid generated by $\{v(i) \mid v\in E^0, i \in \mathbb Z\}$ subject to the relation 
$w(v_j) v_j=\sum_{\{e\in E^1 \mid s(e)=v_j \}} r(\alpha)(1)$, where $1\leq j \leq k$. Then  $S_{k+1}=S_k \langle i,i^{-1}: \overline P \cong \overline{Q(1)}\rangle$ with $P=\textstyle{\bigoplus_{w(v_{k+1})}}S_{k} v_{k+1}$ and $Q =\textstyle{\bigoplus_{\{e \in E^1 \mid s(e)=v_{k+1}\}}}S_k r(e)$.  So one more application of Theorem~\ref{thmbergV} gives that  $\V^{\gr}(S_{k+1})$ is the monoid generated by all the vertices of $E$ subject to relations corresponding to $\{v_1,\dots,v_{k+1}\}$. Thus after repeating this process $m$ times we arrive at the monoid $\mathcal V^{\gr}(R)$ subject to exact same relations of Definition \ref{tal-wei-mon} and that 
\begin{equation}\label{gthfyyd}
    S_m \cong L_\K(E,w).
\end{equation}
Putting these together we have 
\[\mathcal V^{\gr}(L_\K(E,w)) \cong T_{(E,w)},\] 
where the graded $\K$-algebra homomorphism $R\rightarrow L_\K(E,w)$ induces 
 $\mathcal V^{\gr}(R) \rightarrow \mathcal V^{\gr}(L_\K(E,w)), [p_v] \mapsto [L_K(E,w)v]$.

Now suppose $\phi: T_{(E,w)} \rightarrow \mathcal V^{\gr}(A)$ is a pointed $\mathbb Z$-monoid homomorphism. Since $\phi(1_E)=\sum_{v\in E^0}\phi(v)=[A]$, we obtain a graded $A$-module isomorphism $A\cong \bigoplus_{v\in E^0} q_v$, where $q_v$'s are graded finitely generated $A$-modules with $\phi(v)=[q_v]$.   It follows that $A = \bigoplus_{v\in E^0} Ae_v$, where $e_v$'s are pairwise orthogonal idempotents of homogeneous degree zero in $A$, with $Ae_v \cong q_v$ as graded $A$-modules. Thus there is a natural graded $\K$-algebra homomorphism $\eta: R=\prod_{E^0}\K \rightarrow A$, making the ring $A$ a $\mathbb Z$-graded $R\text{-ring}_{\K}$. Since $A \otimes_R p_v \cong q_v$ (use $A \otimes_R R/I \cong A/AI$, for an ideal $I$ of $R$), we obtain the following commutative diagram of $\mathbb Z$-monoids:
 
\[
\begin{tikzcd}
\mathcal V^{\gr}(R)
\arrow[rr, "A \otimes_R -"]
\arrow[dr]
&&
\mathcal V^{\gr}(A)
\\
&
T_{(E,w)}
\arrow[ur, "\phi"']
\end{tikzcd}
\]
Thus 
\[\big[\bigoplus_{w(v)}q_v \big]= w(v)[q_v]=w(v)\phi(v)=\phi(w(v)v)=\phi\big(\sum_{e\in s^{-1}(v)}r(e)(1)\big)=\sum_{e\in s^{-1}(v)}[q_{r(e)}(1)],\] implying graded $A$-module isomorphisms  
\begin{align*}
\bigoplus_{w(v)}q_v &\cong \bigoplus_{e\in s^{-1}(v)}q_{r(e)}(1)\\
A\otimes_R \bigoplus_{w(v)} p_v &\cong A \otimes_R \bigoplus_{e\in s^{-1}(v)}p_{r(e)}(1),
\end{align*}
for vertices $v$ which are not sink. 

Since $L_\K(E,w)$ is the universal ring providing these isomoprhisms (see~\ref{sequi} and \ref{gthfyyd}), it follows that there is a graded $\K$-algebra homomorphism $\psi: L_\K(E,w) \rightarrow A$ such that the following diagram is commutative: 
\[
\begin{tikzcd}
\mathcal V^{\gr}(R)
\arrow[rr, "A\otimes_R -"]
\arrow[dr]
&&
\mathcal V^{\gr}(A)
\\
&
\mathcal V^{\gr}(L_\K(E,w))
\arrow[ur, "\phi = A \otimes_{L_\K(E,w)} -"']
\end{tikzcd}
\]
 This completes the proof.
\end{proof}

We are now in a position to establish the fullness of the graded Grothendieck group functor on the category of all weighted Leavitt path algebras.

\begin{proof}[Proof of Theorem~\ref{weightedlift}]
By Corollary \ref{cor:smashprod}, it follows that
$T_{(E,w)}\cong \mathcal V^{\gr}(L_\K(E,w))$. By Proposition~\ref{goodviasm}(2), $T_{(E,w)}$ is a cancellative monoid, and so 
the positive cone of $K^{\gr}_0(L_\K(E,w))$ is precisely $T_{(E,w)}$. Since $\phi$ is preordered, it induces a $\mathbb Z$-monoid homomorphism  $$\phi: \mathcal V^{\gr}(L_\K(E)) \rightarrow \mathcal V^{\gr}(L_\K(F)).$$ The result now follows from Theorem~\ref{arnonevas}, thus finishing the proof.
\end{proof}

We finish the paper by showing that the $K_0^{\gr}$-group is a complete invariant for Leavitt algebras. Note that in comparison, there is an order isomorphism $K_0(L_\K(n,n+k))\cong K_0(L_\K(n+l,n+k+l))$, for $n,k,l\in \mathbb N$, although $L_\K(n,n+k)\not \cong L_\K(n+l,n+k+l)$.

\begin{thm}\label{thm:classify-gLAs}
The graded Grothendieck group $K_0^{\gr}$ classifies Leavitt algebras $L_\K(n,n+k)$, where $n,k\in \mathbb N$. 
\end{thm}
\begin{proof}
Suppose that there is order preserving $\Z[x,x^{-1}]$-module
isomorphism $K_0^{\gr}(L_\K(n,n+k))\cong K_0^{\gr}(L_\K(n',n'+k'))$. Since the positive cones coincide with the talented monoids, we obtain a $\mathbb Z$-isomorphism $T_{(n,n+k)}\cong T_{(n',n'+k')}$. Now Proposition~\ref{goodviasm} gives an  isomorphism on the level of monogenic monoids $M_{(n,n+k)}\cong M_{(n',n'+k')}$. It is known that every monogenic monoid is completely determined, up to isomorphism, by its index $n$ and period $k$.
So it follows $n=n'$ and $k=k'$.
\end{proof}

%\danger{$K^{\gr}_0$ should be able to classify weighted polycephaly graphs as in \cite{H}.. longer project }

\noindent{{\bf Acknowledgments}}
Roozbeh Hazrat acknowledges Australian Research Council Discovery Project DP230103184. He would like to thank the Vietnam Institute for Advanced Study in Mathematics (VIASM) and Professor T.G. Nam for hosting him in Hanoi. R.J. Damalerio acknowledges the support provided by the DOST–SEI Accelerated Science and Technology Human Resource Development Program (ASTHRDP) and the Centre International de Math\'ematiques Pures et Appliqu\'ees (CIMPA), whose grant supported his research visit to Hanoi. T. G. Nam was partially supported by the Vietnam Academy of Science and Technology under grant CBCLCA.01/26-28.

\end{document}